\documentclass[reqno,a4paper]{amsart}
\usepackage{amssymb}
\usepackage{amsmath}
\newcommand{\angles}[1]{\langle #1 \rangle}

\newcommand{\half}{\frac{1}{2}}
\newcommand{\R}{\mathbb{R}}

\begin{document} 
\newtheorem{prop}{Proposition}[section]
\newtheorem{Def}{Definition}[section] \newtheorem{theorem}{Theorem}[section]
\newtheorem{lemma}{Lemma}[section] \newtheorem{Cor}{Corollary}[section]

\title[Maxwell-Klein-Gordon in Lorenz gauge]{\bf Local well-posedness for low regularity data for the higher-dimensional Maxwell-Klein-Gordon system in Lorenz gauge}
\author[Hartmut Pecher]{
{\bf Hartmut Pecher}\\
Fakult\"at f\"ur  Mathematik und Naturwissenschaften\\
Bergische Universit\"at Wuppertal\\
Gau{\ss}str.  20\\
42119 Wuppertal\\
Germany\\
e-mail {\tt pecher@math.uni-wuppertal.de}}
\date{}

\begin{abstract}

The Cauchy problem for the Maxwell-Klein-Gordon equations in Lorenz gauge in $n$ space dimensions ($n \ge 4$) is shown to be locally well-posed for low regularity (large) data. The result relies on the null structure for the main bilinear terms which was shown to be not only present in Coulomb gauge but also in Lorenz gauge by Selberg and Tesfahun, who proved global well-posedness for finite energy data in three space dimensions. This null structure is combined with product estimates for wave-Sobolev spaces. Crucial for the improvement are the solution spaces introduced by Klainerman-Selberg. Preliminary results were already contained in arXiv:1705.00599.
\end{abstract}
\maketitle
\renewcommand{\thefootnote}{\fnsymbol{footnote}}
\footnotetext{\hspace{-1.5em}{\it 2000 Mathematics Subject Classification:} 
35Q61, 35L70 }

\normalsize 
\setcounter{section}{0}
\section{Introduction and main results}
A classical problem of mathematical physics is the Cauchy problem for the Maxwell-Klein-Gordon system :
\begin{align}
\label{1}
\partial^{\nu} F_{\mu \nu} & =  j_{\mu} \\
\label{2}
D^{(A)}_{\mu} D^{(A)\mu} \phi & = m^2 \phi \, ,
\end{align}
where $m \in \R$ is a constant and
\begin{align}
\label{3}
F_{\mu \nu} & := \partial_{\mu} A_{\nu} - \partial_{\nu} A_{\mu} \\
\label{4}
D^{(A)}_{\mu} \phi & := \partial_{\mu} - iA_{\mu} \phi \\
\label{5}
j_{\mu} & := Im(\phi \overline{D^{(A)}_{\mu} \phi}) = Im(\phi \overline{\partial_{\mu} \phi}) + |\phi|^2 A_{\mu} 
\end{align}
with data
$$\phi(x,0) = \phi_0(x) \quad,\quad \partial_t \phi(x,0) = \phi_1(x) \quad,\quad F_{\mu \nu}(x,0) =  F^0_{\mu \nu}(x)$$  $$A_{\nu}(x,0) = a_{0 \nu}(x)\quad ,\quad \partial_t A_{\nu}(x,0) = \dot{a}_{0 \nu}(x) \, .$$
Here $F_{\mu \nu} : {\mathbb R}^{n+1} \to {\mathbb R}$ denotes the electromagnetic field, $\phi : {\mathbb R}^{n+1} \to {\mathbb C}$ a scalar field and $A_{\nu} : {\mathbb R}^{n+1} \to {\mathbb R}$ the potential. We use the notation $\partial_{\mu} = \frac{\partial}{\partial x_{\mu}}$, where we write $(x^0,x^1,...,x^n) = (t,x^1,...,x^n)$ and also $\partial_0 = \partial_t$ and $\nabla = (\partial_1,...,\partial_n)$. Roman indices run over $1,...,n$ and greek indices over $0,...,n$ and repeated upper/lower indices are summed. Indices are raised and lowered using the Minkowski metric $diag(-1,1,...,1)$.

The Maxwell-Klein-Gordon system describes the motion of a spin 0 particle with mass $m$ self-interacting with an electromagnetic field.

 The potential $A$ is not uniquely determined but one has gauge freedom. The Maxwell-Klein-Gordon equation is namely invariant under the gauge transformation
$\phi \to \phi' = e^{i\chi}\phi$ , $A_{\mu} \to A'_{\mu} = A_{\mu} + \partial_{\mu} \chi$
for any $\chi: {\mathbb R}^{n+1} \to {\mathbb R}$.

Most of the results obtained so far were given in Coulomb gauge $\partial^j A_j = 0$. The seminal paper of Klainerman and Machedon \cite{KM} showed global well-posedness in energy space and above, i.e. for data $\phi_0 \in H^s$ , $\phi_1 \in H^{s-1}$ , $a_{0 \nu} \in H^s$ , $\dot{a}_{0 \nu} \in H^{s-1}$ with $s \ge 1$ in $n=3$ dimensions improving earlier results of Eardley and Moncrief \cite{EM} for smooth data. They used that the nonlinearities fulfill a null condition in the case of the Coulomb gauge. This global well-posedness result was improved by Keel, Roy and Tao \cite{KRT}, who had only to assume $s > \frac{\sqrt 3}{2}$. Local well-posedness for low regularity data was shown for small data by Cuccagna \cite{C} for $s > 3/4$ , where Selberg \cite{S} remarked that the small data assumption can be removed, and finally almost down to $s > 1/2$ , which is the critical regularity with respect to scaling, in the case of small data by Machedon and Sterbenz \cite{MS}, all these results for three space dimensions and in Coulomb gauge. 

In two space dimensions in Coulomb gauge Czubak and Pikula \cite{CP} proved local well-posedness provided that $\phi_0 \in H^s$ , $\phi_1 \in H^{s-1}$ , $a_{0 \nu} \in H^r$ , $\dot{a}_{0 \nu} \in H^{r-1}$,  where $1 \ge s=r > \frac{1}{2}$ or $s=\frac{5}{8}+\epsilon$ , $r=\frac{1}{4}+\epsilon$. 

In four space dimensions Selberg \cite{S} showed local well-posedness in Coulomb gauge for large data and $s>1$. Recently Krieger, Sterbenz and Tataru \cite{KST} showed global well-posedness for data with small energy data ($s=1$) for $n=4$, which is the critical space. For space dimension $n \ge 6$ and small critical Sobolev norm for the data global well-posedness was shown by Rodnianski and Tao \cite{RT}. In general the problem seems to be easier in higher dimensions. In temporal gauge local well-posedness was shown for $n=3$ and small data with $ s > 3/4$ for the more general Yang-Mills equations by Tao \cite{T}.

From now on we exclusively consider the Maxwell-Klein-Gordon equations in Lorenz gauge $\partial^{\mu} A_{\mu} = 0$ ,  which was considered much less in the literature because the nonlinear term $Im(\phi \overline{\partial_{\mu} \phi})$ has no null structure. There is a result by Moncrief \cite{M} in two space dimensions for smooth data, i.e. $s \ge 2$. In three space dimensions the most important progress was made by Selberg and Tesfahun \cite{ST} who were able to circumvent the problem of the missing null condition in the equations for $A_{\mu}$ by showing that the decisive nonlinearities in the equations for $\phi$ as well as $F_{\mu \nu}$ fulfill such a null condition which allows to show that global well-posedness holds for finite energy data, i.e. $\phi_0 \in H^1$, $\phi_1 \in L^2$ , $F^0_{\mu \nu} \in L^2$ , $a_{0\nu} \in \dot{H}^1$ , $\dot{a}_{0 \nu} \in L^2$, and three space dimensions, where $\phi \in C^0({\mathbb R},H^1) \cap C^1({\mathbb R},L^2)$ and $F_{\mu \nu} \in C^0({\mathbb R},L^2)$. The potential possibly loses some regularity compared to the data but as remarked by the authors this is not the main point because one is primarily interested in the regularity of $\phi$ and $F_{\mu \nu}$. Persistence of higher regularity for the solution also holds.

A null structure in Lorenz gauge was first detected for the Maxwell-Dirac system by d'Ancona, Foschi and Selberg \cite{AFS1}.

The present paper is a continuation to dimensions $n \ge 4$ of the author's paper \cite{P}, where local well-posedness was shown for large data in dimension $n=2$ or $n=3$ and (essentially) $ s > \half$ for $n=2$ and $ s > \frac{3}{4}$ for $ n = 3$ .
A preprint of a preceding version was posted in \cite{P1}.  The basis for these results is on the one hand the paper \cite{ST}.  The solution spaces are Bourgain-Klainerman-Machedon spaces $X^{s,b}$. If one wants to lower the regularity further these spaces seem to be not suitable. Thus on the other hand  we rely on the paper by Klainerman-Selberg \cite{KS}, who considered a model problem for Maxell-Klein-Gordon, where the gauge condition is ignored, and where local well-posedness in space dimensions $n \ge 4$ is proved down to the critical range $s > \frac{n}{2}-1$. We treat the full Maxwell-Klein-Gordon system in Lorenz gauge. We appropriately modify the solution spaces in \cite{KS} and obtain local well-posedness for $n \ge 4$ and $s > \frac{n}{2}-\frac{5}{6}$ . More precisely we obtain $\phi \in C^0 ([0,T],H^s(\R^n)) \cap C^1 ([0,T],H^{s-1}(\R^n))$ , $F_{\mu \nu} \in C^0 ([0,T],H^{s-1}(\R^n)) \cap C^1 ([0,T],H^{s-2}(\R^n))$ and $\nabla A_{\mu} , \partial_t A_{\mu} \in  C^0 ([0,T],H^{r-1}(\R^n)) $ for a certain $r$ with $s \ge r > \frac{n}{2}-1$ 
, so that similar to the result in \cite{ST} the potential $A_{\mu}$ is less regular than $\phi$, which is not so important, however. Of course, the proof uses the null structure of the quadratic nonlinear terms with the exception of the term $Im(\phi \overline{\partial_{\mu} \phi})$. It is this term which is responsible for the remaining gap of $\frac{1}{6}$ to the critical regularity $s= \frac{n}{2}-1$. The bilinear estimates in wave-Sobolev spaces were claimed in arbitrary dimension $n$ and proven for $n=2$ and $n=3$ by d'Ancona, Foschi and Selberg in \cite{AFS3} and \cite{AFS2}, respectively. We prove a special case in dimension $n \ge 4$,  based mainly on a result by Klainerman and Tataru \cite{KT}, and also rely on much more general results by Lee-Vargas \cite{LV} (cf. Prop. \ref{LV}), who proved $L^q_t L^r_x$ - estimates for a product of solutions of the wave equation.

We now formulate our main result. We assume the Lorenz condition
\begin{equation}
\label{6}
\partial^{\mu} A_{\mu} = 0
\end{equation}
and Cauchy data
\begin{align}
\label{7}
\phi(x,0) &= \phi_0(x) \in H^s \quad , \quad \partial_t \phi(x,0) = \phi_1(x) \in H^{s-1} \, , 
\end{align}
\begin{align}
\label{8}
F_{\mu \nu}(x,0)& = F^0_{\mu \nu}(x)  \in H^{s-1} \, , 
\end{align}
and
\begin{equation}
\label{9}
A_{\nu}(x,0) = a_{0 \nu}(x) \quad , \quad \partial_t A_{\nu}(x,0) = \dot{a}_{0 \nu}(x) \, ,
\end{equation} 
which fulfill the following conditions
\begin{equation}
\label{10}
a_{00} = \dot{a}_{00} = 0 \, , 
\end{equation}
\begin{equation}
\label{11}
\nabla a_{0j} \in H^{s-1} \, , \,  \dot{a}_{0j} \in H^{s-1} \, ,
\end{equation}
\begin{equation}
\label{12}
\partial^k a_{0k} = 0 \, ,
\end{equation}
\begin{equation}
\label{13}
\partial_j a_{0k} - \partial_k a_{0j} = F^0_{jk} \, ,
\end{equation}
\begin{equation}
\label{14}
\dot{a}_{0k} = F^0_{0k} \, , 
\end{equation}
\begin{equation}
\label{15}
\partial^k F^0_{0k}  = Im(\phi_0 \overline{\phi}_1) \, .
\end{equation}
(\ref{10}) may be assumed because otherwise the Lorenz condition does not determine the potential uniquely. (\ref{12}) follow from the Lorenz condition (\ref{6}) in connection with (\ref{10}). (\ref{13}) follows from (\ref{3}), similarly (\ref{14}) from (\ref{3}) and (\ref{10}). (\ref{1}) requires
$$ \partial^k F^0_{0k} = j_0(0) = Im(\phi_0 \overline{\phi}_1) + |\phi_0|^2 a_{00} = Im(\phi_0 \overline{\phi}_1) \, $$
thus (\ref{15}). By (\ref{12}) we have
$$ \Delta a_{0j} = \partial^k \partial_k a_{0j} = \partial^k(\partial^j a_{0k} - F^0_{jk}) = - \partial^k F^0_{jk} \, , $$
so that $a_{0j}$ is uniquely determined as
$$ a_{0j} = D^{-2} \partial^k F^0_{jk} $$
and fulfills (\ref{11}).

We define the wave-Sobolev spaces $H^{s,b}$ as the completion of  $\mathcal{S}({\mathbb R}^{n+1})$ with respect to the norm
$$ \|u\|_{H^{s,b}} =  \| \langle \xi \rangle^s \langle  |\tau| - |\xi| \rangle^b \widehat{u}(\tau,\xi) \|_{L^2_{\tau \xi}}  $$
and $H^{s,b}[0,T]$ as the space of the restrictions to $[0,T] \times \mathbb{R}^n$.
 Similarly the space $\dot{H}^{s,b}$ is defined, where $\langle \xi \rangle$ is replaced by $|\xi|$ .

Let $\Lambda^{\alpha}$ , $D^{\alpha}$ , $D_-^{\alpha}$ , $D_+^{\alpha}$ and $\Lambda_+^{\alpha}$ be the multipliers with symbols $\langle \xi \rangle^{\alpha}$, 
$|\xi|^{\alpha}$ ,
 $ ||\tau|-|\xi||^{\alpha}$ , $ (|\tau |+|\xi|)^{\alpha}$ and $\langle |\tau|+|\xi| \rangle^{\alpha}$ , respectively, where $ \langle \, \cdot \, \rangle = (1+|\, \cdot \,|^2)^{\frac{1}{2}}$ .

$\Box = \partial_t^2 - \Delta$ is the d'Alembert operator. $a+ = a+\epsilon$ for a sufficiently small $\epsilon > 0$ . 

We also use the notation $ u \precsim v $ , if $|\widehat{u}| \lesssim \widehat{v}$ . \\[1em]
Our main theorem reads as follows:
\begin{theorem}
\label{Theorem1}
Let $n \ge 4$ . Assume $s > \frac{n}{2}-\frac{5}{6} $ , $r > \frac{n}{2}-1$ and $s \ge r \ge s-\half$ , $3s-2r > \frac{n-1}{2}$ , $2r-s > \frac{n-3}{2}$ . 
 The data are assumed to fulfill (\ref{7}) - (\ref{15}). Then the problem (\ref{1}) - (\ref{6}) has a unique local solution
$$ \phi \in H^{s,\frac{1}{2}+}[0,T] \, , \, \partial_t \phi \in H^{s-1,\frac{1}{2}+}[0,T] $$
and
$$ F_{\mu \nu} \in H^{s-1,\frac{1}{2}+}[0,T] \, , \, \partial_t F_{\mu \nu} \in H^{s-2,\frac{1}{2}+}[0,T]$$
relative to a potential $A=(A_0,A_1,...,A_n)$ with
$D A ,$ $ \partial_t A \in H^{r-1,\half+}[0,T]$ .
\end{theorem}
\noindent{\bf Remarks:} 
\begin{enumerate}
\item We immediately obtain $\phi \in C^0([0,T],H^s) \cap C^1([0,T],H^{s-1}) $ , and $F_{\mu \nu} \in C^0([0,T],H^{s-1}) \cap C^1([0,T],H^{s-2}) $ .
\item The case $  \frac{n}{2}-\frac{5}{6}+\delta \, , \, r = \frac{n}{2}-1+\delta $ 
 is admissible, where $\delta > 0$ is an arbitrary number.
\end{enumerate}

The paper is organized as follows: in chapter II we prove the null structure of $A^{\mu} \partial_{\mu} \phi$ and in the Maxwell part. In chapter III the local well-posedness result for (\ref{16}),(\ref{17}) is formulated (Theorem \ref{Theorem2}). In chapter IV the main tools for its proof are introduced, namely the Strichartz type estimates, the bilinear estimates in wave-Sobolev spaces (Prop. \ref{Prop.3.6}) and the bilinear estimates for the solutions of the wave equation by Lee-Vargas (Prop. \ref{LV}). Here also the solution spaces are defined and its fundamental properties are formulated, which were proven by Klainerman-Selberg \cite{KS}. In Prop. \ref{Prop1.6} it is shown that it is possible to reduce local well-posedness for the Cauchy problem for a system of wave equations to estimates for the nonlinearities in these spaces. In chapter V we prove Theorem \ref{Theorem2}. In chapter VI we prove our main theorem (Theorem \ref{Theorem1}). We show that the Maxwell-Klein-Gordon system is satisfied and $F_{\mu \nu}$ has the desired regularity properties using the null structure of the Maxwell part and again Strichartz type estimates and the bilinear estimates in wave-Sobolev spaces again.

\section{Null structure}
We can reformulate the system (\ref{1}),(\ref{2}) under the Lorenz condition (\ref{6}) as follows:
$$\square \,A_{\mu} = \partial^{\nu} \partial_{\nu} A_{\mu} = \partial^{\nu}(\partial^{\mu} A_{\nu} - F_{\mu \nu}) = -\partial^{\nu} F_{\mu \nu} = - j_{\mu} \, , $$
thus (using the notation $\partial = (\partial_0,\partial_1,...,\partial_n)$):
\begin{equation}
\label{16}
\square \,A = -Im (\phi \overline{\partial \phi}) - A|\phi|^2 =: N(A,\phi) 
\end{equation}
and
\begin{align*}
m^2 \phi & = D^{(A)}_{\mu} D^{(A)\mu} \phi = \partial_{\mu} \partial^{\mu} \phi -iA_{\mu} \partial^{\mu} \phi -i\partial_{\mu}(A^{\mu} \phi) - A_{\mu}A^{\mu} \phi \\
& = \square \,\phi - 2i A^{\mu} \partial_{\mu} \phi - A_{\mu} A^{\mu} \phi \,
\end{align*}
thus
\begin{equation}
\label{17}
\square \, \phi = 2i A^{\mu} \partial_{\mu} \phi + A_{\mu} A^{\mu} \phi + m^2 \phi =: M(A,\phi) \, .
\end{equation}
Conversely, if $\square\,A_{\mu} = -j_{\mu}$ and $F_{\mu \nu} := \partial_{\mu} A_{\nu} - \partial_{\nu} A_{\mu}$ and the Lorenz condition (\ref{6}) holds then
$$ \partial^{\nu} F_{\mu \nu} = \partial^{\nu}(\partial_{\mu} A_{\nu} - \partial_{\nu} A_{\mu}) = \partial_{\mu} \partial^{\nu} A_{\nu} - \partial^{\nu} \partial_{\nu} A_{\mu} = -\square \,A_{\mu} = j_{\mu} \, $$
thus (\ref{1}),(\ref{2}) is equivalent to (\ref{16}),(\ref{17}), if (\ref{3}),(\ref{4}) and (\ref{6}) are satisfied. \\[0.5em]
\noindent{\bf Null structure of $A^{\mu} \partial_{\mu}\phi$.} 
We define
$$ Q_{\alpha \beta}(\phi,\psi) := \partial_{\alpha} \phi \partial_{\beta} \psi - \partial_{\beta} \phi \partial_{\alpha} \psi  $$
and the Riesz transform $R_k := D^{-1} \partial_k $ and $\mathbf{A} = (A_1,...,A_n)$. We use the decomposition 
$$ \mathbf{A} = A^{df} + A^{cf} \, ,$$
where
$$ A^{df}_j = R^k(R_j A_k - R_k A_j) \quad , \quad A^{cf}_j = - R_j R_k A^k \, , $$
so that
$$ A^{\mu} \partial_{\mu} \phi  = P_1 + P_2 = (-A^0 \partial_t \phi + A^{cf} \cdot \nabla \phi) + A^{df} \cdot \nabla \phi \, .$$
If the Lorenz gauge $\partial_k A^k = \partial_t A^0$ is satisfied, we obtain
\begin{align}
\nonumber
 P_1 &= -A_0 \partial_t \phi - R_j R_k A^k \partial^j \phi \\
\nonumber
& = -A_0 \partial_t \phi - (D^{-2} \nabla\partial_t A_0) \cdot \nabla\phi \\
\nonumber
& = \partial_j(D^{-1} R^j A_0) \partial_t \phi -  \partial_t(D^{-1} R_j A_0)  \partial^j \phi \\
\label{***}
& = -Q_{0j}(D^{-1} R^j A_0,\phi) \, .
\end{align}
Moreover 
\begin{align}
\nonumber
P_2 = A^{df} \cdot \nabla \phi &= R^k(R_j A_k - R_k A_j)\partial^j \phi \\
\nonumber
&= \Lambda^{-2} \partial^k \partial_j A_k \partial^j \phi + A_j \partial^j \phi \\
\nonumber
&= -\half \left( \Lambda^{-2}(\partial_j \partial^j A_k - \partial_j \partial_k A^j)\partial^k \phi - \Lambda^{-2}(\partial^k \partial_j A_k - \partial_k \partial^k A_j) \partial^j \phi \right) \\
\nonumber
&= -\half \left(\partial_j \Lambda^{-1}(R^j  A_k - R_k A^j)\partial^k \phi - \partial^k \Lambda^{-1}(R_j A_k - R_k  A_j) \partial^j \phi \right) \\
\label{****}
&= - \half Q_{jk}(\Lambda^{-1}(R^j A^k - R^k A^j),\phi) \, .
\end{align}
\\[0.5em]
{\bf Null structure in the Maxwell part.}
We start from Maxwell's equations (\ref{1}), i.e. $-\partial^0 F_{l0} + \partial^k F_{lk} = j_l$ and $\partial^k F_{0k} = j_0$ and obtain
\begin{align}
\nonumber
\square F_{k0} & = -\partial_0(\partial_0 F_{k0}) + \partial^l \partial_l F_{k0} \\ 
\nonumber
& = -\partial_0(\partial^l F_{kl} - j_k) + \partial^l \partial_l F_{k0} \\
\nonumber
& = -\partial^l \partial_0(\partial_k A_l - \partial_l A_k) + \partial_0 j_k + \partial^l \partial_l F_{k0} \\
\nonumber
& = -\partial^l[\partial_k(\partial_0 A_l - \partial_l A_0) - \partial_l(\partial_0 A_k - \partial_k A_0)] + \partial_0 j_k + \partial^l \partial_l F_{k0} \\
\nonumber
& = -\partial^l \partial_k F_{0l} + \partial^l \partial_l F_{0k} + \partial_0 j_k + \partial^l \partial_l F_{k0} \\
\nonumber
& =  -\partial^l \partial_k F_{0l} +\partial_0 j_k \\
\nonumber
& = -\partial_k j_0 + \partial_0 j_k \\
\nonumber
&= \, Im(\partial_0 \phi \overline{\partial_k \phi}) + Im(\phi \overline{\partial_0 \partial_k \phi}) + \partial_0(A_k |\phi|^2) \\
 \nonumber
&\quad- Im(\partial_k \phi \overline{\partial_0 \phi}) - Im(\phi \overline{\partial_k \partial_0 \phi}) - \partial_k(A_0 |\phi|^2) \\
\nonumber
&= Im(\partial_t \phi\overline{\partial_k \phi} - \partial_k \phi \overline{\partial_t \phi}) + \partial_t(A_k |\phi|^2) - \partial_k(A_0 |\phi|^2) \\
\label{2.10a}
& = Im \,Q_{0k}(\phi,\overline{\phi}) + \partial_t(A_k |\phi|^2) - \partial_k(A_0 |\phi|^2) 
\end{align}
and
\begin{align}
\nonumber
\square F_{kl} & = -\partial_0 \partial_0 F_{kl} + \partial^m \partial_m F_{kl} \\\nonumber
& = -\partial_0 \partial_0 (\partial_k A_l - \partial_l A_k) + \partial^m \partial_m F_{kl} \\\nonumber
& = -\partial_0 \partial_k(\partial_0 A_l - \partial_l A_0) + \partial_0 \partial_l (\partial_0 A_k - \partial_k A_0) + \partial^m \partial_m F_{kl} \\\nonumber
& = -\partial_0 \partial_k F_{0l} + \partial_0 \partial_l F_{0k} + \partial^m \partial_m F_{kl} \\\nonumber
& = \partial_k \partial_0 F_{l0} - \partial_l \partial_0 F_{k0} + \partial^m \partial_m F_{kl} \\\nonumber
& = \partial_k (\partial^m F_{lm} - j_l) - \partial_l(\partial^m F_{km} - j_k)+ \partial^m \partial_m F_{kl} \\\nonumber
& = \partial_k \partial^m F_{lm} - \partial_l \partial^m F_{km} + \partial^m \partial_m F_{kl} +\partial_l j_k - \partial_k j_l \\\nonumber
& = \partial_k \partial^m(\partial_l A_m - \partial_m A_l) - \partial_l \partial^m(\partial_k A_m - \partial_m A_k)+ \partial^m \partial_m F_{kl} +\partial_l j_k - \partial_k j_l \\\nonumber
& = \partial^m \partial_m (\partial_l A_k - \partial_k A_l) + \partial^m \partial_m F_{kl} +\partial_l j_k - \partial_k j_l \\\nonumber
& = \partial^m \partial_m F_{lk} + \partial^m \partial_m F_{kl} +\partial_l j_k - \partial_k j_l \\
\nonumber
&   = \partial_l j_k - \partial_k j_l \\
\nonumber
&=  Im(\partial_l \phi \overline{\partial_k \phi}) + Im(\phi \overline{\partial_l \partial_k \phi}) + \partial_l(A_k |\phi|^2) \\
 \nonumber
&\quad- Im(\partial_k \phi \overline{\partial_l \phi}) - Im(\phi \overline{\partial_k \partial_l \phi}) - \partial_k(A_l |\phi|^2) \\
\nonumber
&= Im(\partial_l \phi\overline{\partial_k \phi} - \partial_k \phi \overline{\partial_l \phi}) + \partial_l(A_k |\phi|^2) - \partial_k(A_l |\phi|^2) \\
\label{2.11a}
& = Im \,Q_{lk} (\phi,\overline{\phi}) + \partial_l(A_k |\phi|^2) - \partial_k(A_l |\phi|^2)
\end{align}

Fundamental is the following result for the null form $Q(u,v)$ , where $Q(u,v)$ denotes either $Q_{0i}(u,v)$ or $Q_{ij}(u,v)$ , which was given in \cite{KMBT}, Prop. 1.
\begin{align*}
 Q(u,v)  & \lesssim D_+^{\half} D_-^{\half} (D_+^{\half} u D_+^{\half} v) + D_+^{\half}(D_+^{\half} D_-^{\half} u D_+^{\half} v) + D_+^{\half}(D_+^{\half} u D_+^{\half} D_-^{\half} v) 
\end{align*}
Interpolation with the trivial bound $D_+ u D_+ v$ gives (for $0 \le \epsilon \le \frac{1}{4}$) :
\begin{align}
\label{Q1}
Q(u,v) &  \lesssim D_+^{\half-2\epsilon} D_-^{\half-2\epsilon} (D_+^{\half+2\epsilon} u D_+^{\half+2\epsilon} v) \\
\nonumber
&+ D_+^{\half-2\epsilon}(D_+^{\half+2\epsilon} D_-^{\half-2\epsilon} u D_+^{\half+2\epsilon} v) + D_+^{\half-2\epsilon}(D_+^{\half+2\epsilon} u D_+^{\half+2\epsilon} D_-^{\half-2\epsilon} v) \, ,
\end{align}
and by the fractional Leibniz rule for $D_+$  we obtain
\begin{align}
\label{Q2}
Q(u,v)  &  \lesssim D_-^{\half-2\epsilon} (D_+^{\half+2\epsilon} u D_+ v)
 +D_+^{\half+2\epsilon} D_-^{\half-2\epsilon} u D_+ v                          \\ \nonumber
&+ D_+^{\half+2\epsilon} u D_-^{\half-2\epsilon}  D_+ v
+ D_-^{\half-2\epsilon}(D_+ u D_+^{\half+2\epsilon} v) \\ \nonumber
 &+ D_+D_-^{\half-2\epsilon} u D_+^{\half+2\epsilon}v
+ D_+ u D_+^{\half+2\epsilon} D_-^{\half-2\epsilon}v
 \, .
\end{align}

\section{Local well-posedness}
 From (\ref{7}) we have $\phi_0 \in H^s$ , $\phi_1 \in H^{s-1}$ , and from (\ref{11}) we have for $r\le s$: $\nabla a_{0j} \in H^{r-1}$ , $\dot{a}_{0j} \in H^{r-1}$. \\[0.3em]

Our aim is to show the following local well-posedness result:
\begin{theorem}
\label{Theorem2}
Let $n \ge 4$. Assume 
$$ s > \frac{n}{2}-\frac{5}{6} \, , \, r > \frac{n}{2}-1 \, , \, s \ge r-1 \, , \, r \ge s-\half \, , \,2r-s > \frac{n-3}{2} \, , \,3s-2r > \frac{n-1}{2} \,.$$
 Let $\phi_0 \in H^s$, $\phi_1 \in H^{s-1}$ , and $D a_0 \, , \, \dot{a}_0 \in H^{r-1}$  be given. Then the system 
\begin{align}
\label{16'}
\square \,A &= -Im (\phi \overline{\partial \phi}) - A|\phi|^2 =: N(A,\phi) \\
\label{17'}
\square \, \phi &= -2i(A_0 \partial_t \phi + D^{-2}\nabla \partial_t A_0 \cdot \nabla \phi - A^{df} \cdot \nabla \phi) + A_{\mu} A^{\mu} \phi + m^2 \phi =: \tilde{M}(A,\phi) 
\end{align}
with initial conditions
$\phi(0) = \phi_0 \, , \, (\partial_t \phi)(0)=\phi_1\, , \,  A(0)= a_0 \, , \, (\partial_t A)(0) = \dot{a}_0$ has a unique local solution
$$\phi \in H^{s,\frac{1}{2}+}[0,T] \, , \, \partial_t \phi \in H^{s-1,\frac{1}{2}+}[0,T] 
\, , \quad
D A \, , \, \partial_t A \in F^{r-1}_T \, .$$  
Here the space $F^{r-1}_T$ is defined in Definition \ref{Def} below.
\end{theorem}
\noindent{\bf Remark:} We recall that we have shown above that $M(A,\phi) = \tilde{M}(A,\phi)$ in Lorenz gauge. Thus the solution of (\ref{16'}),(\ref{17'}) also solves (\ref{16}),(\ref{17}), if the Lorenz condition holds.

\section{Preliminaries}
We start by recalling the Strichartz type estimates for the wave equation in the next proposition.
\begin{prop}
\label{Str}
If $n \ge 2$ and
\begin{equation}
2 \le q \le \infty, \quad
2 \le r < \infty, \quad
\frac{2}{q} \le (n-1) \left(\half - \frac{1}{r} \right), 
\end{equation}
then the following estimate holds:
$$ \|u\|_{L_t^q L_x^r} \lesssim \|u\|_{H^{\frac{n}{2}-\frac{n}{r}-\frac{1}{q},\half+}} \, ,$$
especially in the case $n \ge 4$ :
$$ \|u\|_{L^2_t L^r_x} \lesssim \|u\|_{H^{\frac{n-1}{2}-\frac{n}{r},\half +}} $$
for $ \frac{2(n-1)}{n-3} \le r < \infty $ .
\end{prop}
\begin{proof}
This is the Strichartz type estimate, which can be found for e.g. in \cite{GV}, Prop. 2.1, combined with the transfer principle.
\end{proof}

\noindent{\bf Remark:} The $H^{s,b}$-space may be replaced by the homogeneous $\dot{H}^{s,b}$-space. \\[0.3em]

An immediate consequence is the following modified Strichartz estimate.
\begin{prop}
\label{mStr}
If $n \ge 4$ , $ 2 \le r \le \frac{2(n-1)}{n-3}$ , one has the estimate
$$ \|u\|_{L^2_t L^r_x} \lesssim \|u\|_{H^{\frac{n+1}{2}\left(\half-\frac{1}{r}\right),\frac{n-1}{2}\left(\half-\frac{1}{r}\right)+}} \lesssim  \|u\|_{H^{\frac{n+1}{2}\left(\half-\frac{1}{r}\right),\half+}} \, . 
$$
\end{prop}
\begin{proof}
The last estimate is trivial. For the first one we
interpolate the trivial identity $\|u\|_{L^2_t L^2_x} = \|u\|_{H^{0,0}} $ with the estimate
$$ \|u\|_{L^2_t L^{\frac{2(n-1)}{n-3}}_x} \lesssim \|u\|_{H^{\frac{n+1}{2(n-1)},\half+}} \, , $$
which holds by Prop. \ref{Str}. 
\end{proof}
 The following proposition was proven by \cite{KT}.
\begin{prop}
\label{TheoremE}
Let $n \ge 2$, and let $(q,r)$ satisfy:
$$
2 \le q \le \infty, \quad
2 \le r < \infty, \quad
\frac{2}{q} \le (n-1) \left(\half - \frac{1}{r} \right)
$$
Assume that
\begin{gather*}
0 < \sigma < n - \frac{2n}{r} - \frac{4}{q}, \\
s_1, s_2 < \frac{n}{2} - \frac{n}{r} - \frac{1}{q}, \\
s_1 + s_2 + \sigma = n - \frac{2n}{r} - \frac{2}{q}. \end{gather*} 
then
$$
\|D^{-\sigma}(uv)\|_{L_t^{q/2} L_x^{r/2}}
\lesssim
\|u\|_{H^{s_1,\half+}}
\|v\|_{H^{s_2,\half+}} \, .
$$
\end{prop}

Fundamental for the proof of Theorem \ref{Theorem2} are  bilinear estimates in wave-Sobolev spaces which were proven by d'Ancona, Foschi and Selberg in the cases $n=2$ in \cite{AFS2} and $n=3$ in \cite{AFS3} and claimed for arbitrary dimensions in a general form.
We now prove a special case of these bilinear estimate, which holds in higher dimensions.

\begin{prop}
\label{Prop.3.6}
Assume $n \ge 4$ and
$$s_0+s_1+s_2 > \frac{n-1}{2} \, , \, (s_0+s_1+s_2)+s_1+s_2 > \frac{n}{2} \, , \, s_0+s_1 \ge 0 \, , \, s_0+s_2 \ge 0 \, , \, s_1+s_2 \ge 0  . $$
The following estimate holds:
$$ \|uv\|_{H^{-s_0,0}} \lesssim \|u\|_{H^{s_1,\half+}} \|v\|_{H^{s_2,\half+}} \, . $$
\end{prop}
\begin{proof}
We have to prove
\begin{align*}
I & := \int_* \frac{\widehat{u}_1(\xi_1,\tau_1)}{\angles{\xi_1}^{s_1} \angles{|\xi_1|-|\tau_1|}^{\half+}}  \frac{\widehat{u}_2(\xi_2,\tau_2)}{\angles{\xi_2}^{s_2} \angles{|\xi_2|-|\tau_2|}^{\half+}}  \frac{\widehat{u}_0(\xi_0,\tau_0)}{\angles{\xi_0}^{s_0}} \lesssim \|u_1\|_{L^2_{xt}}
\|u_2\|_{L^2_{xt}} \, .
\end{align*}
Here * denotes integration over $\xi_0+\xi_1+\xi_2=0 $ and $\tau_0+\tau_1+\tau_2=0$ . Remark, that we may assume that the Fourier transforms are nonnegative. We consider different regions. \\
1. If $|\xi_0| \sim |\xi_1| \gtrsim |\xi_2|$ and $s_2 \ge 0$, we obtain
\begin{align*}
I \sim \int_* \frac{\widehat{u}_1(\xi_1,\tau_1)}{\angles{\xi_1}^{s_1+s_0} \angles{|\xi_1|-|\tau_1|}^{\half+}}  \frac{\widehat{u}_2(\xi_2,\tau_2)}{\angles{\xi_2}^{s_2} \angles{|\xi_2|-|\tau_2|}^{\half+}}  \widehat{u}_0(\xi_0,\tau_0) \, .
\end{align*}
Thus we have to show
$$ \|uv\|_{L^2_{xt}} \lesssim \|u\|_{H^{s_1+s_0,\half+}} \|v\|_{H^{s_1,\half+}} \, . $$
By Prop. \ref{Str} we obtain
$$ \|uv\|_{L^2_{xt}} \lesssim \|u\|_{L^{\infty}_t L^2_x} \|v\|_{L^2_t L^{\infty}_x} \lesssim \|u\|_{H^{0,\half+}} \|v\|_{H^{\frac{n-1}{2}+,\half+}} $$
and also
$$ \|uv\|_{L^2_{xt}} \lesssim  \|u\|_{H^{\frac{n-1}{2}+,\half+}}      \|v\|_{H^{0,\half+}} \, .$$
Bilinear interpolation gives for $0 \le \theta \le 1 $ :
$$ \|uv\|_{L^2_{xt}} \lesssim \|u\|_{H^{\frac{n-1}{2}(1-\theta)+,\half+}} \|v\|_{H^{\frac{n-1}{2}\theta+,\half+}}  \, , $$
so that 
$$   \|uv\|_{L^2_{xt}} \lesssim \|u\|_{H^{s_1+s_0,\half+}} \|v\|_{H^{s_2,\half+}} \, , $$
if $ s_0+s_1+s_2 > \frac{n-1}{2} $ and $s_1+s_0 \ge 0$ .\\
2. If $|\xi_0| \sim |\xi_2| \gtrsim |\xi_1|$ and $s_1 \ge 0$ , we obtain similarly
$$  \|uv\|_{L^2_{xt}} \lesssim \|u\|_{H^{s_2+s_0,\half+}} \|v\|_{H^{s_1,\half+}} \, , $$
if $ s_0+s_1+s_2 > \frac{n-1}{2} $ and $s_2+s_0 \ge 0$ .\\
3. If $|\xi_1| \ge |\xi_2|$ , $s_0 \le 0$ and $s_2 \ge 0$, we have $|\xi_0| \lesssim |\xi_1| $ , so that
$ \angles{\xi_0}^{-s_0} \lesssim \angles{\xi_1}^{-s_0}$ and we obviously obtain the same result as in 1. \\
4. If $|\xi_1| \le |\xi_2|$ , $s_0 \le 0$ and $s_1 \ge 0$ , we obtain the same result as in 2. \\
5.  If $|\xi_0| \sim |\xi_1| \gtrsim |\xi_2|$ and $s_2 \le 0$ we obtain
$$ I \lesssim \int_* \frac{\widehat{u}_1(\xi_1,\tau_1)}{\angles{\xi_1}^{s_0+s_1+s_2} \angles{|\xi_1|-|\tau_1|}^{\half+}}  \frac{\widehat{u}_2(\xi_2,\tau_2)}{\angles{|\xi_2|-|\tau_2|}^{\half+}}  \widehat{u}_0(\xi_0,\tau_0)  \lesssim \|u_1\|_{L^2_{xt}} \|u_2\|_{L^2_{xt}} \, , $$
because under our asumption $s_0+s_1+s_2 > \frac{n-1}{2}$ we obtain by Prop. \ref{Str}:
$$ \|uv\|_{L^2_{xt}} \le \|u\|_{L^2_t L^{\infty}_x} \|v\|_{L^{\infty}_x L^2_t} \lesssim \| u \|_{H^{s_0+s_1+s_2,\half+}} \|v\|_{H^{0,\half+}} \, . $$
6. If $|\xi_1| \ge |\xi_2|$ , $s_2 \le 0$ and $s_0 \le 0$ , or \\
7. If $|\xi_0| \sim |\xi_2| \ge |\xi_1|$ and $s_1 \le 0$ , or \\
8. If $|\xi_1| \le |\xi_2|$ , $s_1 \le 0$ and $s_0 \le 0$ , the same argument applies.

Thus we are done, if $s_0 \le 0$ , and also, if $s_0 \ge 0$ , and $|\xi_0| \sim |\xi_2 | \ge |\xi_1|$ or $ |\xi_0| \sim |\xi_1 | \ge |\xi_2|$ . 

It remains to consider the following case: $|\xi_0| \ll |\xi_1| \sim |\xi_2|$ and $s_0 > 0$ . We apply Prop. \ref{TheoremE} which gives
$$ \|uv\|_{H^{-s_0,0}} \lesssim \|u\|_{H^{s_1,\half+}} \|v\|_{H^{s_2,\half+}} \, , $$
under the conditions $0 < s_0 < \frac{n}{2} - 1 $ , $s_0+s_1+s_2 = \frac{n-1}{2} $ and $s_1,s_2 < \frac{n-1}{4}$ . The last condition is not necessary in our case $|\xi_1| \sim |\xi_2|$ . Remark that this implies $s_1+s_2 > \half$ , so that $s_0+s_1+s_2+s_1+s_2 > \frac{n}{2}$ .  , The second condition can now be replaced by $s_0+s_1+s_2 \ge \frac{n-1}{2}$ , because we consider inhomogeneous spaces.

Finally we consider the case $|\xi_0| \ll |\xi_1| \sim |\xi_2|$ and $s_0 \ge \frac{n}{2}-1 $ .
If $s_0 > \frac{n}{2} $ and $s_1+s_2 \ge 0$ we obtain the claimed estimate by Sobolev
$$ \|uv\|_{H^{-s_0,0}} \lesssim \|u\|_{H^{0,\half+}} \|v\|_{H^{0,\half+}} \le  \|u\|_{H^{s_1,\half+}} \|v\|_{H^{s_2,\half+}} \, . $$
We now interpolate the special case
$$ \|uv\|_{H^{-\frac{n}{2}-,0}} \lesssim \|u\|_{H^{0,\half+}} \|v\|_{H^{0,\half+}}  $$
with the following estimate
$$ \|uv\|_{H^{1-\frac{n}{2}+,0}} \lesssim \|u\|_{H^{\frac{1}{4}+,\half+}} \|v\|_{H^{\frac{1}{4}+,\half+}} \, , $$
which follows from Prop. \ref{TheoremE} . We obtain
$$\|uv\|_{H^{-s_0-,0}} \lesssim \|u\|_{H^{k+,\half+}} \|v\|_{H^{k+,\half+}} \, , $$
where $s_0 = (1-\theta)\frac{n}{2} - \theta(1-\frac{n}{2}) = \frac{n}{2} - \theta \, \Leftrightarrow \, \theta = \frac{n}{2}-s_0$ , $0 \le \theta \le 1 $ , $k=\frac{\theta}{4} = \frac{n}{8}-\frac{s_0}{4}.$  Using our asumption  $(s_0+s_1+s_2)+s_1+s_2 > \frac{n}{2} \, \Leftrightarrow \, \frac{n}{2}-s_0 < 2(s_1+s_2),$  we obtain $0 \le k < \frac{s_1+s_2}{2}$ .  Because $|\xi_1| \sim |\xi_2|$ , we obtain
$$ \|uv\|_{H^{-s_0,0}} \lesssim \|u\|_{H^{s_1,\half+}} \|v\|_{H^{s_2,\half+}} $$
for $(s_0+s_1+s_2)+s_1+s_2 > \frac{n}{2}$ and $s_1+s_2 \ge 0$ .
\end{proof}

\begin{Cor}
\label{Cor.3.1}
Under the assumptions of Prop. \ref{Prop.3.6}
$$ \|uv\|_{H^{-s_0,0}} \lesssim \|u\|_{H^{s_1,\half-}} \|v\|_{H^{s_2,\half-}}  \, . $$
\end{Cor}
\begin{proof}
This follows by bilinear interpolation of the estimate of Prop. \ref{Prop.3.6} with the estimate
$$\|uv\|_{H^{N,0}} \lesssim \|u\|_{H^{N,\frac{1}{4}+}} \|v\|_{H^{N,\frac{1}{4}+}}  \, , $$
where, say, $N > \frac{n}{2}$ , which follows by Sobolev apart from the special case $s_1=-s_2,$  in which we interpolate with the estimate
$$\|uv\|_{H^{-N,0}} \lesssim \|u\|_{H^{N,\frac{1}{4}+}} \|v\|_{H^{-N,\frac{1}{4}+}}   $$
in order to save the condition $s_1=-s_2$ .
\end{proof}

The following multiplication law is well-known:
\begin{prop} {\bf (Sobolev multiplication law)}
\label{SML}
Let $n\ge 2$ , $s_0,s_1,s_2 \in \mathbb{R}$ . Assume
$s_0+s_1+s_2 > \frac{n}{2}$ , $s_0+s_1 \ge 0$ ,  $s_0+s_2 \ge 0$ , $s_1+s_2 \ge 0$. Then the following product estimate holds:
$$ \|uv\|_{H^{-s_0}} \lesssim \|u\|_{H^{s_1}} \|v\|_{H^{s_2}} \, .$$
\end{prop}

Next we formulate a special case of the fundamental estimates for the $L^q_t L^p_x$-norm of the product of solutions of the wave equation due to Lee-Vargas \cite{LV}.
\begin{prop}
\label{LV}
Let $n \ge 4$ . \\
1. Assume
$$ \Box \,f = \Box \,g = 0 $$
in $\R^n \times \R$ .
The estimate
$$ \|D^{\beta_0} (fg)\|_{L^q_t L^2_x} \lesssim (\|f(0)\|_{\dot{H}^{\alpha_1}} + \|(\partial_t f)(0)\|_{\dot{H}^{\alpha_1-1}})(\|g(0)\|_{\dot{H}^{\alpha_2}} + \|(\partial_t g)(0)\|_{\dot{H}^{\alpha_2-1}}) $$
holds, provided $1< q \le 2$ and
\begin{align}
\label{LV2}
\frac{1}{q} &=  \frac{n}{2}-\alpha_1 - \alpha_2 +\beta_0  \, , \\
\label{LV5}
\beta_0 & > \frac{2}{q} - \frac{n+1}{2} \\
\label{LV12}
\alpha_1,\alpha_2 & <  \frac{n}{2} + \half - \frac{2}{q} \, .
\end{align}
2. If (\ref{LV2},(\ref{LV5}),(\ref{LV12}) and $\alpha_1,\alpha_2 \ge 0$ are satisfied, the following estimate holds
$$ \|uv\|_{L^q_t L^2_x} \lesssim \|u\|_{\dot{H}^{\alpha_1,\half+}} \|v\|_{\dot{H}^{\alpha_2,\half+}} \, . $$
\end{prop}
\begin{proof}
This follows by easy calculations from \cite{LV}, Theorem 1.1. combined with the transfer principle.
\end{proof}

We now come to the definition of the solution spaces, which are very similar to the spaces introduced by \cite{KS}. We prepare this by defining a modification of the standard $L^q_t L^r_x$-spaces. 
\begin{Def}
\label{Def.1.2}
If $1 \le q, r \le
\infty$, $u \in {\mathcal S}'$ and $\widehat u$ is a tempered function, set $$
\|u\|_{{\mathcal L}_t^q {\mathcal L}_x^r} = \sup \left\{ \int_{\R^{1+n}}  |\widehat u (\tau,\xi)| \widehat v(\tau,\xi) \, d\tau d\xi : v \in {\mathcal S} , \widehat v \ge 0 ,
 \|v\|_{L^{q'}_t L^{r'}_x} = 1 \right\}, $$
where $1 = \frac{1}{q} + \frac{1}{q'}$ and $1 = \frac{1}{r}
+ \frac{1}{r'}$. Let ${\mathcal L}_t^q {\mathcal L}_x^r$ be the corresponding subspace of
${\mathcal S}'$.
\end{Def}

This is a translation invariant norm and it only
depends on the size of the Fourier transform.  Observe that 
${\mathcal L}^2_t {\mathcal L}^2_x = L^2_t L^2_x$  and
$$
 \|u\|_{{\mathcal L}_t^q {\mathcal L}_x^r} \le \|u\|_{L^{q}_t L^{r}_x}\quad \text{whenever} \quad
\widehat u \ge 0. 
$$

\begin{Def}
\label{Def} Let $r,s > \frac{n}{2}-1$ .
Our solution spaces are defined as follows: \\
$\{ (u,v) \in F^{r-1}  \times G^s \} \, , $ where \\
\begin{align*}
\|u\|_{F^{r-1}} &:= \|\Lambda_+ u\|_{H^{r-2,\half+\epsilon}} +\| \Lambda_+^{\half+2\epsilon} \Lambda^{-2 +3\epsilon} \Lambda_-^{\half} u\|_{{\mathcal L}_t^1 {\mathcal L}_x^{\infty}} \\
\|v\|_{G^s} &:= \|\Lambda_+ v\|_{H^{s-1,\half+\epsilon}}  \ ,
\end{align*}
where $\epsilon > 0$ is sufficiently small. $F^{r-1}_T$ and $G^s_T$ denotes the restriction to the time interval $[0,T]$.
\end{Def}
This is a Banach space (\cite{KS}, Prop. 4.2).

Next we recall some fundamental properties of the ${\mathcal L}_t^q {\mathcal L}_x^r$-spaces, which were given by \cite{KS}, starting with a H\"older-type estimate.

\begin{prop}\label{Prop.4.3} Suppose $\frac{1}{q} =
\frac{1}{q_{1}} + \frac{1}{q_{2}}$ and $\frac{1}{r} = \frac{1}{r_{1}} +
\frac{1}{r_{2}}$, where the $q$'s and $r$'s all belong to $[1,\infty]$.
Then $$
\|uv\|_{{\mathcal L}^q_t {\mathcal L}^r_x}  \le \|u\|_{{\mathcal L}^{q_1}_t {\mathcal L}^{r_1}_x}  \| v\|_{L^{q_2}_t  L^{r_2}_x} \, .
$$
for all $v$  with $\widehat
v \ge 0$.
\end{prop}

The following duality argument holds. 
\begin{prop}\label{Prop.4.5} Let $1 \le a,b,q,r \le
\infty$.
{
\renewcommand{\theenumi}{\alph{enumi}}
\renewcommand{\labelenumi}{(\theenumi)}
\begin{enumerate}
\item If
\begin{equation}\label{DualityA}
\|G\|_{L^{a'}_t L^{b'}_x} \lesssim \| \Lambda^{\alpha}
\Lambda_{-}^{\beta} G\|_{L^{q'}_t L^{r'}_x} 
\end{equation} 
for all $G$ ,
 then
$$
\|F\|_{L^{q}_t L^{r}_x} \lesssim \| \Lambda^{\alpha}
\Lambda_{-}^{\beta} F\|_{L^{a}_t L^{b}_x} 
$$
for all $F$ .
\item If \eqref{DualityA} holds for all $G$ with $\widehat G
\ge 0$, then
$$
\|F\|_{{\mathcal L}^{q}_t {\mathcal L}^{r}_x} \lesssim \| \Lambda^{\alpha}
\Lambda_{-}^{\beta} F\|_{{\mathcal L}^{a}_t {\mathcal L}^{b}_x} \, .
$$
for all $F$.
\end{enumerate}
}
\end{prop}
\begin{proof}
\cite{KS}, Proposition 4.5. 
\end{proof}

The next proposition shows that a Sobolev type embedding also carries over to the ${\mathcal L}_t^q {\mathcal L}_x^r$-spaces.
\begin{prop}
\label{Cor.4.6} Let $1 \le a,b,q,r \le \infty$ , $\alpha,\beta \in {\mathbb R}$ .
If
$$ \|\Lambda^{\alpha} \Lambda_{-}^{\beta}u\|_{L^q_t L^r_x} \lesssim \|u\|_{ L^a_t L^b_x}  $$
for all $u$ with $\widehat{u} \ge 0$ , then
$$\|\Lambda^{\alpha} \Lambda_{-}^{\beta}u\|_{{\mathcal L}^q_t {\mathcal L}^r_x} \lesssim \|u\|_{ {\mathcal L}^a_t {\mathcal L}^b_x } \, . $$
\end{prop}
\begin{proof} \cite{KS}, Cor. 4.6.
\end{proof}

The following result is also fundamental for the proof of our main theorem.
\begin{prop}\label{Prop.4.8} Let $ n \ge 4$ . If $\frac{2(n-1)}{n-3} \le r < \infty $,  $s = \frac{n}{2} -\frac{n}{r}-\half $ and $\theta > \half$ ,
then
$$ \|u\|_{{\mathcal L}^1_t {\mathcal L}^{r}_x}
\lesssim   \|\Lambda^s \Lambda_-^{\theta} u\|_{{\mathcal L}^{1}_t {\mathcal L}^2_x} \, .$$
\end{prop}
\begin{proof} \cite{KS}, Lemma 4.8.
\end{proof}

\begin{Cor}
\label{Cor.4.2}
 Let $ n \ge 4$ and $\theta > \half$ ,
then
$$ \|u\|_{{\mathcal L}^1_t {\mathcal L}^{\infty }_x}
\lesssim   \|\Lambda^{\frac{n}{2}-\half+} \Lambda_-^{\theta} u\|_{{\mathcal L}^{1}_t {\mathcal L}^2_x} \, .$$
\end{Cor}
\begin{proof} This follows from Prop. \ref{Prop.4.8} and  Prop. \ref{Cor.4.6}.
\end{proof}

We also need an elementary estimate which is used as a tool for replacing $H^{s,-\half+}$-norms by $H^{s,-\half-}$-norms.
\begin{lemma}\label{Lemma8.10} Let $\alpha,\beta
\ge 0$. Then
\begin{align*}
\Lambda_{-}^{-\beta}(uv) &\precsim \Lambda_-^{-\alpha-\beta} (\Lambda_+^{\alpha} u \Lambda_+^{\alpha}v) \, , \\
\Lambda_{-}^{-\beta}(uv) &\precsim
\Lambda_{-}^{-\alpha-\beta}(u \Lambda^{\alpha} v) + u \Lambda^{-\beta} v
\end{align*}
for all $u$ and $v$ with $\widehat u, \widehat v \ge 0$. \end{lemma}
\begin{proof}
\cite{KS}, Lemma 8.10.
\end{proof}

Finally, we formulate the fundamental theorem which allows to reduce the local well-posedness for a system of nonlinear wave equations to suitable estimates for the nonlinearities. It is also essentially contained in  \cite{KS}. 

\begin{prop}
\label{Prop1.6}
Let $Du_0 \in H^{r-1}$ , $u_1 \in H^{r-1}$ , $v_0 \in H^s$ , $v_1 \in H^{s-1}$ be given. Assume that for $0 < T < 1$ :
\begin{align*}
\| \Lambda_+^{-1} \Lambda_-^{\epsilon-1}D_+ {\mathcal M}(u,\partial u,v , \partial v)\|_{F^{r-1}_T} & \le \omega_1(\|D_+ u\|_{F^{r-1}},\|v\|_{G^s}) \, , \\
\| \Lambda_+^{-1} \Lambda_-^{\epsilon-1} {\mathcal N}(u,\partial u,v , \partial v)\|_{G^s_T} & \le \omega_2(\|D_+u\|_{F^{r-1}},\|v\|_{G^s}) \, , 
\end{align*}
and
\begin{align*}
&\| \Lambda_+^{-1} \Lambda_-^{\epsilon-1} D_+({\mathcal M}(u,\partial u,v , \partial v) - {\mathcal M}(u',\partial u',\partial v'))\|_{F^{r-1}_T} \\
&+
\| \Lambda_+^{-1} \Lambda_-^{\epsilon-1} ({\mathcal N}(u,\partial u,v , \partial v) - {\mathcal N}(u',\partial u',v' , \partial v'))\|_{G^s_T} \\ &\le \omega(\|D_+u\|_{F^{r-1}},\|D_+u'\|_{F^{r-1}},\|v\|_{G^s}, \|v'\|_{G^s}) (\|D_+(u-u')\|_{F^{r-1}} + \|v-v'\|_{G^s}) \, , \\
\end{align*}
where $\omega,\omega_1,\omega_2$ are continuous functions with $\omega(0,0,0,0) = \omega_1(0,0) = \omega_2(0,0) = 0$.
Then the Cauchy problem
$$ \Box \, u = {\mathcal M}(u,\partial u,v,\partial v) \quad , \quad \Box \, v = {\mathcal N}(u,\partial u,v,\partial v) $$
with data
$$u(0) = u_0 \, , \, (\partial_t u)(0) = u_1 \, , \, v(0)= v_0 \, , \,(\partial_t v)(0) = v_1 $$
is locally well-posed, i.e. , there exists $T>0$ , such that there exists a unique solution $D_+u \in F^{r-1}_T$ , $v \in G^s_T$ .
\end{prop}
\begin{proof}
This is proved by the contraction mapping principle provided the solution space fulfills suitable assumptions. The case of a single equation $\Box \, u = {\mathcal M}(u,\partial u)$ and the solution space ${\mathcal X}^s$ for $s > \frac{n}{2}-1$, given by the norm
$\|u\|_{{\mathcal X}^s} = \|\Lambda_+ u\|_{H^{s-1,\half+\epsilon}} + \|\Lambda^{\gamma} \Lambda_-^{\half} u\|_{{\mathcal L}_t^1 {\mathcal L}_x^{2n}} $ , $\gamma > 0$ small , was proven by \cite{KS}, Theorems 5.4 and 5.5, Propositions 5.6 and 5.7. Our case is a straightforward modification of their results, thus we omit the proof.
We just remark that the only modification in the case of our solution space is the following estimate in the proof of \cite{KS}, Prop. 5.6:
\begin{align*}
\| \Lambda_+^{\half+2\epsilon} \Lambda^{-2 +3\epsilon} \Lambda_-^{\half} u \|_{{\mathcal L}^1_t {\mathcal L}^{\infty}_x} & \lesssim \| \Lambda_+^{\half+2\epsilon} \Lambda^{\frac{n}{2}-\frac{5}{2}+} \Lambda_-^{1+}  u \|_{{\mathcal L}^{1}_t {\mathcal L}^{2}_x} \\
& \lesssim \| {\mathcal F} ( \Lambda_+^{\half +2\epsilon} \Lambda^{\frac{n}{2}-\frac{5}{2}+} \Lambda_-^{1+} u )\|_{L^2_{\xi} L^{\infty}_{\tau}} \\
& \lesssim \| {\mathcal F} ( \Lambda_+ \Lambda^{r-2} \Lambda_-^{1+} u )\|_{L^2_{\xi} L^{\infty}_{\tau}} \, .
\end{align*}
The first estimate follows from Corollary \ref{Cor.4.2}, and the last estimate holds by our assumption $ r > \frac{n}{2}-1$ .
\end{proof}

\section{Proof of Theorem \ref{Theorem2}}

An inspection of the nonlinear terms in (\ref{16'}) and (\ref{17'}) - using (\ref{***}) and (\ref{****}) and noting the fact that the Riesz transforms 
$R_i$ are bounded in the spaces involved - reduces the estimates in Proposition \ref{Prop1.6} 
to proving (we remark, that due to the multilinear character of the nonlinearity the estimates for the difference can be treated exactly like the other estimates): \\
The estimate for the null forms $Q \in \{Q_{0i},Q_{ij}\}$ :
\begin{equation}
  \label{29}
  \| \Lambda_+^{-1} \Lambda_-^{\epsilon -1} Q(D^{-1}A, \phi)\|_{G^s_T } \lesssim \|D_+A\|_{F^{r-1}} \|\phi\|_{G^s}\, ,
\end{equation}
the following estimate for the other quadratic term
\begin{align}
  \label{35} 
  \| \Lambda_+^{-1} \Lambda_-^{\epsilon -1} \partial_{\mu} (\phi \partial_{\nu} \psi) \|_{F^{r-1}_T}
     &\lesssim \|\phi\|_{G^s} \|\partial_{\nu} \psi\|_{G^{s-1}}
     \end{align}
and the following estimates for the cubic terms:
\begin{align}
\label{39}
\| \Lambda_+^{-1} \Lambda_-^{\epsilon-1} \partial_{\mu} (A \phi \psi)\|_{F^{r-1}_T} & \lesssim \|D_+A\|_{F^{r-1}} \|\phi\|_{G^s} \|\psi\|_{G^s}\, ,\\
\label{40}
\| \Lambda_+^{-1} \Lambda_-^{\epsilon-1}  (A_{\mu} A_{\nu} \phi)\|_{G^s_T} & \lesssim \|D_+A_{\mu}\|_{F^{r-1}}\|D_+A_{\nu}\|_{F^{r-1}}\|\phi\|_{G^s} \, .
\end{align}
Here $ 0 < T < 1 $ . \\
\noindent
{\bf Important remarks:} 1. We assume in the following that the Fourier transforms of $u$ and $v$ are nonnegative. This means no loss of the generality, because the norms involved in the desired estimates do only depend on the size of the Fourier transforms. \\
2. In all those estimates which contain only $H^{s,b}$-norms we may replace $Du$ by $\Lambda u$ by \cite{T}, Cor. 8.2. \\
{\bf Proof of (\ref{29}):} By (\ref{Q1}) we reduce to the following estimates
$$
\|D^{-1} D_+^{\half+2\epsilon}u D_+^{\half+2\epsilon}v\|_{H^{s-\half-2\epsilon,0}}  \lesssim \|D_+u\|_{H^{r-1,\half+\epsilon}} \|\Lambda_+v\|_{H^{s-1,\half+\epsilon}} \, ,
$$
which follows from
\begin{align}
\label{29.1}
\|uv\|_{H^{s-\half-2\epsilon,0}} & \lesssim \|D^{\frac{3}{2}-2\epsilon}u\|_{H^{r-1,\half+\epsilon}} \|v\|_{H^{s-\half-2\epsilon,\half+\epsilon}} \, ,
\end{align}
and
\begin{align}
\label{29.2}
\|uv\|_{H^{s-\half-2\epsilon,-\half+2\epsilon}} & \lesssim \|D_+^{\half-2\epsilon} \Lambda_-^{-\half+2\epsilon} Du\|_{F^{r-1}} \|v\|_{H^{s-\half-2\epsilon,\half+\epsilon}} \\
\label{29.3}
\|uv\|_{H^{s-\half-2\epsilon,-\half+2\epsilon}} & \lesssim \|D^{\frac{3}{2}-2\epsilon}u\|_{H^{r-1,\half+\epsilon}} \|v\|_{H^{s-\half-2\epsilon,0}} \, .
\end{align}
(\ref{29.1}): We remark that $D^{\frac{3}{2}-2\epsilon}$ may be replaced by $\Lambda^{\frac{3}{2}-2\epsilon}$ by \cite{T}, Cor. 8.2. It then follows from Prop. \ref{Prop.3.6} with parameters $s_0=\half-s+2\epsilon$, $s_1=r+\half-2\epsilon$ , $s_2=s-\half-2\epsilon$, so that $s_0+s_1+s_2 = r + \half -2 \epsilon > \frac{n-1}{2}$ and $s_1+s_2 = r+s-4\epsilon > n-2 \ge 2$ for $r,s > \frac{n}{2}-1$ and $n \ge 4$ . \\
(\ref{29.3}) follows similarly with $s_0=s-\half-2\epsilon$ , $s_1= \half-s+2\epsilon$ , $s_2 =r+\half-2\epsilon$ , so that $s_1+s_2 = r-s+1 \ge \half$ under the assumption $r \ge s- \half$ . \\
(\ref{29.2}) is handled as follows. By the fractional Leibniz rule we reduce to two estimates. \\
The first estimate is
$$ \| (\Lambda^{s-\half-2\epsilon} u) v\|_{H^{0,-\half+2\epsilon}} \lesssim \|D^{\frac{3}{2}-2\epsilon}u\|_{H^{r-1,0}} \|v\|_{H^{s-\half-2\epsilon,\half + \epsilon}} \, , $$
which follows from Prop. \ref{Prop.3.6} with $s_0= r-s+1$ , $s_1= 0$, $s_2= s-\half-2\epsilon$ , so that $s_0+s_1+s_2 = r + \half-2\epsilon > \frac{n-1}{2}$ and $s_1+s_2=s-\half-2\epsilon > \frac{n-3}{2} \ge \half$ for $n \ge 4$.\\
The second estimate is
$$ \|u \Lambda^{s-\half-2\epsilon}v\|_{H^{0,-\half+2\epsilon}} \lesssim \|D_+^{\half-2\epsilon} \Lambda_-^{-\half+2\epsilon} Du\|_{F^{r-1}} \|v\|_{H^{s-\half-2\epsilon,\half+\epsilon}} \, . $$
We obtain by Lemma \ref{Lemma8.10}:
\begin{align*}
&\|u \Lambda^{s-\half-2\epsilon} v\|_{H^{0,-\half+2\epsilon}} \\
 &\quad \quad\lesssim \| \Lambda_-^{-\half-\frac{\epsilon}{2}} (\Lambda^{\frac{5}{2}\epsilon} u \Lambda^{s-\half-2\epsilon}v)\|_{H^{0,0}} + \| \Lambda^{-\half+2\epsilon} u \Lambda^{s-\half-2\epsilon} v\|_{H^{0,0}} \, .
\end{align*}
The last term is easily estimated by Prop. \ref{Prop.3.6} under the assumption $ r > \frac{n}{2} -1 $ :
$$ \| \Lambda^{-\half+2\epsilon} u \Lambda^{s-\half-2\epsilon} v\|_{H^{0,0}} \lesssim \|Du\|_{H^{r-\half-2\epsilon,0}} \|v\|_{H^{s-\half-2\epsilon,\half+\epsilon}} \,. $$
For the  first term we obtain:
\begin{align*}
\| \Lambda^{\frac{5}{2}\epsilon} u \Lambda^{s-\half-2\epsilon} v\|_{H^{0,-\half-\frac{\epsilon}{2}}} & \lesssim \|\Lambda^{\frac{5}{2}\epsilon} u \Lambda^{s-\half-2\epsilon} v\|_{{\mathcal L}^1_t {\mathcal L}_x^2} \\
& \lesssim \|\Lambda^{\frac{5}{2}\epsilon} u\|_{{\mathcal L}^1_t {\mathcal L}_x^{\infty}} \| \Lambda^{s-\half-2\epsilon} v\|_{{\mathcal L}^{\infty}_t {\mathcal L}_x^2} 
\end{align*}
If $u$ has large frequencies we estimate
\begin{align*}
\|\Lambda^{\frac{5}{2}\epsilon} u\|_{{\mathcal L}^1_t {\mathcal L}_x^{\infty}} &\lesssim \|\Lambda^{\frac{5}{2}\epsilon}\Lambda^{-\frac{3}{2}+2\epsilon}  D_+^{\half-2\epsilon} D u\|_{{\mathcal L}^1_t {\mathcal L}_x^{\infty}} \\
&\lesssim \|\Lambda_+^{\half+2\epsilon} \Lambda^{-2+3\epsilon} \Lambda_-^{\half} (D_+^{\half-2\epsilon} \Lambda_-^{-\half+2\epsilon} D u)\|_{{\mathcal L}^1_t {\mathcal L}_x^{\infty}} \\
 &= \| D_+^{\half-2\epsilon} \Lambda_-^{-\half+2\epsilon} Du\|_{F^{r-1}}
\end{align*}
and
$$\| \Lambda^{s-\half-2\epsilon} v\|_{{\mathcal L}^{\infty}_t {\mathcal L}_x^2} \lesssim \|v\|_{H^{s-\half-2\epsilon,\half+\epsilon}} \, .$$
If $u$ has low frequencies we crudely estimate
\begin{align*}
\| \Lambda^{\frac{5}{2}\epsilon} u \Lambda^{s-\half-2\epsilon} v\|_{H^{0,-\half-\frac{\epsilon}{2}}} & \lesssim \| \Lambda^{\frac{5}{2}\epsilon} u \Lambda^{s-\half-2\epsilon} v\|_{L^2_t L^2_x} \\
& \lesssim \| \Lambda^{\frac{5}{2}\epsilon} u\|_{L^2_t L^{\infty}_x} \| \Lambda^{s-\half-2\epsilon} v\|_{L^{\infty}_t L^2_x} \\
& \lesssim \| D^{\frac{n}{2}-\epsilon}  u\|_{L^2_t L^2_x} \|v\|_{H^{s-\half-2\epsilon,\half+\epsilon}} \\
& \lesssim \|Du\|_{L^2_t L^2_x} \|v\|_{H^{s-\half-2\epsilon,\half+\epsilon}} \\
& \lesssim \|Du\|_{H^{r-\half-2\epsilon,0}} \|v\|_{H^{s-\half-2\epsilon,\half+\epsilon}} \, .
\end{align*}
This completes the proof of (\ref{29}). \\
[0.5em]
{\bf Proof of (\ref{35}):}
{\bf A.} We start with the estimate for the first part of the $F^{r-1}$-norm. It suffices to prove
$$ \|uv\|_{H^{r-1,-\half+2\epsilon}} \lesssim \|u\|_{H^{s,\half+\epsilon}} \|v\|_{H^{s-1,\half+\epsilon}} \, . $$
By the fractional Leibniz rule this reduces to the following estimates:
\begin{align}
\label{35.1}
\|uv\|_{H^{0,-\half+2\epsilon}} &
\lesssim \|u\|_{H^{s-r+1,\half+\epsilon}} \|v\|_{H^{s-1,\half+\epsilon}} \, , \\
\label{35.2}
\|uv\|_{H^{0,-\half+2\epsilon}} &\lesssim \|u\|_{H^{s,\half+\epsilon}} \|v\|_{H^{s-r,\half+\epsilon}} \, .
\end{align}
If $2s-r > \frac{n}{2}-\half$ both estimates follow from Prop. \ref{Prop.3.6} even with left hand side $\|uv\|_{H^{0,0}}$. \\
If $2s-r \le \frac{n}{2}-\half$ we apply Prop. \ref{LV} with parameters $\alpha_1= s-r+1$ , $\alpha_2=s-1$ for (\ref{35.1}) and with $\alpha_1 =s$ , $\alpha_2=s-r$ for (\ref{35.2}), where we used that the left hand side is majorized by $\|uv\|_{L^q_t L^2_x}$ for any $1<q \le 2$. We have to show the conditions: \\
(\ref{LV2}): The condition $\alpha_1+\alpha_2 = 2s-r = \frac{n}{2}-\frac{1}{q}$ can be fulfilled with some $1<q\le 2$,  if $\frac{n}{2}-\half \ge 2s-r > \frac{n}{2}-1$ . Using our assumptions   $3s-2r >\frac{n}{2}-\half$ and $r \ge s-\half$ this holds, because $2s-r = (3s-2r) + (r-s) > \frac{n}{2}-1$ . \\
(\ref{LV12}): The decisive condition is $s < \frac{n}{2}+\half-\frac{2}{q} = \frac{n}{2}+\half+4s-2r-n= 4s-2r - \frac{n}{2}+\half,$ which is equivalent to our assumption $3s-2r>\frac{n}{2}-\half$ . \\
{\bf B.} Next we consider the second part of the $F^{r-1}$-norm. It suffices to prove
$$ \| \Lambda_+^{-\half+2\epsilon} \partial_{\nu} \Lambda^{-2+3\epsilon} \Lambda_-^{-\half+\epsilon} (uv)\|_{{\mathcal L}^1_t[0,T] {\mathcal L}^{\infty}_x} \lesssim \|\Lambda_+ u\|_{H^{s-1,\half+\epsilon}} \|\Lambda_+v\|_{H^{s-2,\half+\epsilon}} \, . $$
We use $\Lambda_-^{2\epsilon}u \precsim \Lambda^{2\epsilon}u$ and Cor. \ref{Cor.4.6},  so that the left hand side is bounded by
\begin{align*}
 \| \partial_{\nu}  \Lambda^{-\frac{5}{2}+7\epsilon} \Lambda_-^{-\half-\epsilon} (uv)\|_{{\mathcal L}^1_t[0,T] {\mathcal L}^{\infty }_x} & \lesssim  \| \partial_{\nu}  \Lambda^{\frac{n}{2}-3+8\epsilon}  (uv)\|_{{\mathcal L}^1_t[0,T] {\mathcal L}^2_x}  \\
& \lesssim \|  \partial_{\nu} \Lambda^{\frac{n}{2}-3+8\epsilon}  (uv)\|_{ L^1_t[0,T] L^2_x}\, ,
\end{align*}
where we used Cor. \ref{Cor.4.2}.  \\
Case a: In the case of high frequencies of the product $uv$ it suffices to prove for $T \le 1$ :
\begin{align*}
&\|  D^{\frac{n}{2}-3+8\epsilon}  ((\partial_{\nu}u)v)\|_{ L^1_t[0,T] L^2_x} +  \|  D^{\frac{n}{2}-3+8\epsilon}  (u(\partial_{\nu}v))\|_{ L^1_t[0,T] L^2_x}\\
& \quad \quad \quad \quad \quad\lesssim \|u\|_{H^{s,\half+\epsilon}} \|v\|_{H^{s-1,\half+\epsilon}} \, . 
\end{align*}
The first term is estimated by H\"older in time replacing the  $L^1_t[0,T] L^2_x$-norm by the $L^q_t L^2_x$-norm for an arbitrary $ 1 \le q \le 2$ first and then applying Prop. \ref{LV} with parameters $\alpha_1 = \alpha_2 = \frac{n}{2}-\frac{11}{6} +5\epsilon$ and $\beta_0 = \frac{n}{2}-3+8\epsilon$ . We check the necessary conditions: \\
(\ref{LV2}): $\frac{n}{2}-\alpha_1-\alpha_2 +\beta_0 =  \frac{2}{3} -2\epsilon =:\frac{1}{q}$. \\
(\ref{LV5}): $\beta_0 = \frac{n}{2}-3+8\epsilon > \frac{2}{q} - \frac{n+1}{2} = \frac{4}{3}-\frac{n+1}{2} \, \Leftrightarrow n > \frac{23}{6}-8\epsilon$ , which is true for $n \ge 4$ . \\
(\ref{LV12}): The estimate $\frac{n}{2}-\frac{11}{6} + 5\epsilon < \frac{n}{2}+\half - \frac{2}{q} = \frac{n}{2}-\frac{5}{6}+4\epsilon$ is fulfilled. \\
For the second term  we argue similarly in the case of high frequencies of $v$ , namely choosing in Prop. \ref{LV} the parameters $\alpha_1 = \frac{n}{2}-\frac{5}{6}$ , $\alpha_2= \frac{n}{2} - \frac{17}{6}+10\epsilon$,  $\beta_0 = \frac{n}{2}-3+8\epsilon$ (here $\alpha_2$ is negative for $ n \le 5 $ , therefore the high frequency assumption for $v$). The conditions (\ref{LV2}) and (\ref{LV5}) are checked as before. \\
(\ref{LV12}): The decisive condition is $\alpha_1 < \frac{n}{2} - \frac{5}{6}+4\epsilon$ , which is fulfilled. \\
In the case of low frequencies of $v$ we estimate
\begin{align*}
 \|  \Lambda^{\frac{n}{2}-3+8\epsilon}  (u(\partial_{\nu}v))\|_{ L^1 L^2_x} & \lesssim \|  (\Lambda^{\frac{n}{2}-3+8\epsilon} u) \partial_{\nu}v\|_{ L^1 L^2_x} \\
& \lesssim\| \Lambda^{\frac{n}{2}-3+8\epsilon} u \|_{L^2_t L^2_x} \|\partial_{\nu}v\|_{ L^2 L^{\infty}_x} \\
& \lesssim \|\Lambda_+ u\|_{H^{s-1,\half+\epsilon}} \|\Lambda_+ v\|_{L^2_t H^{\frac{n}{2}+}_x} \\
& \lesssim  \|\Lambda_+ u\|_{H^{s-1,\half+\epsilon}} \|\Lambda_+ v\|_{H^{s-2,\half+\epsilon }} \, .
\end{align*}
Case b: If the frequencies of $uv$ are low, we easily obtain for arbitrarily large $N$ by Prop. \ref{SML} :
\begin{align*}
 \|  \partial_{\nu} \Lambda^{\frac{n}{2}-3+8\epsilon}  (uv)\|_{ L^1_t L^2_x} & \lesssim
\|\partial_{\nu} (uv)\|_{L^1_t H^{-N}_x} \\
& \lesssim \|(\partial_{\nu} u)v\|_{L^1_t H^{-N}_x} + \|u(\partial_{\nu}v)\|_{L^1_t H^{-N}_x} \\
& \lesssim \|D_+ u\|_{L^2_t H^{s-1}_x} \|v\|_{L^2_t H^{s-1}_x} + \|u\|_{L^2_t H^s_x} \|D_+v\|_{L^2_t H^{s-2}_x} \\
& \lesssim \|\Lambda_+ u\|_{H^{s-1,\half+\epsilon}} \|\Lambda_+ v\|_{H^{s-2,\half+\epsilon}} \, .
\end{align*}
This completes the proof of (\ref{35}). \\[0.3em]

It remains to consider the cubic nonlinearities.\\
{\bf Proof of (\ref{40}):} The following estimate holds:
$$ \|A_{\mu} A_{\nu} \phi \|_{H^{s-1,-\half+2\epsilon}} \lesssim \|DA_{\mu}\|_{H^{r-1,\half+\epsilon}} \|DA_{\nu}\|_{H^{r-1,\half+\epsilon}} \|\phi\|_{H^{s,\half+\epsilon}} $$
a. If $s > \frac{n}{2} - \half$ we apply Prop. \ref{Prop.3.6} and Cor. \ref{Cor.3.1} twice, where we remark again, that we may replace $D$ by $\Lambda$, to obtain
\begin{align*}
\|A_{\mu} A_{\nu} \phi \|_{H^{s-1,-\half+2\epsilon}} & \lesssim \|D(A_{\mu} A_{\nu})\|_{H^{s-2,0}} \|\phi\|_{H^{s,\half+\epsilon}} \\
& \lesssim (\|(DA_{\mu}) A_{\nu}\|_{H^{s-2,0}} + \|A_{\mu} (DA_{\nu})\|_{H^{s-2,0}}) \|\phi\|_{H^{s,\half+\epsilon}} \\
& \lesssim \|DA_{\mu}\|_{H^{r-1,\half+\epsilon}} \| DA_{\nu}\|_{H^{r-1,\half+\epsilon}} \|\phi\|_{H^{s,\half+\epsilon}} \, ,
\end{align*}
where we choose for the first step $s_0=s-1$ , $s_1 = 1-s$ , $s_2=s$ , so that $s_0+s_1+s_2= s > \frac{n}{2}-\half$ by assumption and $s_1+s_2=1$ . For the last step the choice $s_0=2-s$ , $s_1=r-1$ , $s_2=r$ gives $s_0+s_1+s_2 = 2r-s+1 > \frac{n-1}{2}$ by our assumption and also $s_1+s_2 = 2r-1 >n-3 \ge 1$ . \\
b. If $s \le \frac{n}{2}-\half$ we obtain similarly
\begin{align*}
\|A_{\mu} A_{\nu} \phi \|_{H^{s-1,-\half+2\epsilon}} & \lesssim \|D(A_{\mu} A_{\nu})\|_{H^{\frac{n}{2}-\frac{5}{2}+,0}} \|\phi\|_{H^{s,\half+\epsilon}} \\
& \lesssim (\|(DA_{\mu}) A_{\nu}\|_{H^{\frac{n}{2}-\frac{5}{2}+,0}} + \|A_{\mu} (DA_{\nu})\|_{H^{\frac{n}{2}-\frac{5}{2}+,0}})  \|\phi\|_{H^{s,\half+\epsilon}} \\
& \lesssim \|DA_{\mu}\|_{H^{r-1,\half+\epsilon}} \| DA_{\nu}\|_{H^{r-1,\half+\epsilon}} \|\phi\|_{H^{s,\half+\epsilon}} \, .
\end{align*}
For the first step we choose $s_0=\frac{n}{2}-\frac{3}{2}+$ , $s_1= s$ , $s_2=1-s$ , so that $s_0+s_1+s_2 = \frac{n}{2}-\half + $ and $s_1+s_2=1$ , whereas for the last step $s_0=\frac{5}{2}-\frac{n}{2}-$ , $s_1=r-1$ , $s_2=r$,  so that $s_0+s_1+s_2 = 2r + \frac{3}{2}-\frac{n}{2}- > \frac{n}{2}-\half $ for $r > \frac{n}{2}-1$ , and $s_1+s_2 = 2r-1 > 1$.  \\[0.5em]
{\bf Proof of (\ref{39}):} {\bf A.} We start with the first part of the $F^{r-1}$-norm. It suffices to show
$$ \| A \phi \psi\|_{H^{r-1,-\half+2\epsilon}} \lesssim \|DA\|_{H^{r-1,\half + \epsilon}} \|\phi\|_{H^{s,\half+\epsilon}} \|\psi\|_{H^{s,\half+\epsilon}} $$
We use Prop. \ref{Prop.3.6} and Cor. \ref{Cor.3.1} and obtain
\begin{align*}
\|A \phi \psi\|_{H^{r-1,-\half+2\epsilon}} &\lesssim \|DA\|_{H^{r-1,\half+\epsilon}} \|\phi \psi\|_{H^{\frac{n}{2}-\frac{3}{2}+,0}} \\
&\lesssim \|DA\|_{H^{r-1,\half+\epsilon}} \|\phi\|_{H^{s,\half+\epsilon}} \|\psi\|_{H^{s,\half+\epsilon}} \, , 
\end{align*}
where for the first step $s_0=\frac{n}{2}-\frac{3}{2}+$ , $s_1=r$ , $s_2= 1-r$ , so that $s_0+s_1+s_2 = \frac{n}{2}-\half+$ and $s_1+s_2=1$, and for the second step $s_0=\frac{3}{2}
-\frac{n}{2}-$ , $s_1=s_2=s$ , so that $s_0+s_1+s_2 = \frac{3}{2}-\frac{n}{2}+2s- > \frac{n}{2}-\half$ and $s_1+s_2=2s > n-2 \ge 2$ by our assumption $s > \frac{n}{2}-1$ . \\
{\bf B.} Next we estimate the second part of the $F^r$-norm.
It suffices to prove
\begin{align*}
\| & \Lambda_+^{-\half+2\epsilon} \Lambda^{-2+3\epsilon} \Lambda_-^{-\half+\epsilon} \partial_{\nu}(uvw)\|_{{\mathcal L}^1_t {\mathcal L}^{\infty}_x}\\
& \quad \quad\lesssim \|D_+u\|_{H^{r-1,\half+\epsilon}} \|\Lambda_+ v\|_{H^{s-1,\half+\epsilon}}  \|\Lambda_+ w\|_{H^{s-1,\half+\epsilon}} \, . 
\end{align*}
We use $\Lambda_-^{2\epsilon}u \precsim \Lambda^{2\epsilon}u$ and Cor. \ref{Cor.4.2}, so that the left hand side is bounded by
\begin{align*}
 &\|  \Lambda^{-\frac{5}{2}+7\epsilon} \Lambda_-^{-\half-\epsilon} \partial_{\nu}(uvw)\|_{{\mathcal L}^1_t {\mathcal L}^{\infty }_x} \lesssim  \|  \Lambda^{\frac{n}{2}-3+8\epsilon} \partial_{\nu} (uvw)\|_{{\mathcal L}^1_t {\mathcal L}^2_x}  \, .
\end{align*}
1. Let us first consider the case $n \ge 6$. By the fractional Leibniz rule we have to consider the following terms (by symmetry in $v$ and $w$). \\
a. We obtain
\begin{align*}
 \| (\Lambda^{\frac{n}{2}-3+8\epsilon} \partial_{\nu} u)vw\|_{L^1_t L^2_x} & \lesssim \|  \Lambda^{\frac{n}{2}-3+8\epsilon} \partial_{\nu} u\|_{L^{\infty}_t L^{\frac{2n}{n-2}}_x} \|v\|_{L^2_t L^{2n}_x} \|w\|_{L^2_t L^{2n}_x} \\
& \lesssim \|\Lambda^{\frac{n}{2}-2+8\epsilon} \partial_{\nu}u\|_{L^{\infty}_t L^2_x } \|v\|_{H^{\frac{n}{2}-1,\half+\epsilon}} \|w\|_{H^{\frac{n}{2}-1,\half+\epsilon}} \\
& \lesssim \|D_+u\|_{H^{r-1,\half+\epsilon}} \|v\|_{H^{s,\half+\epsilon}}  \|w\|_{H^{s,\half+\epsilon}} \, ,
\end{align*}
where we used H\"older, Sobolev, Prop. \ref{Str} and the assumptions $r,s > \frac{n}{2}-1$ . \\
b. Similarly we obtain
\begin{align*}
 \|(\partial_{\nu}u) (\Lambda^{\frac{n}{2}-3+8\epsilon}v)w\|_{L^1_t L^2_x} & \lesssim  \|\partial_{\nu} u\|_{L^{\infty}_t L^{\frac{n}{2}}_x}| \|\Lambda^{\frac{n}{2}-3+8\epsilon} v\|_{L^2_t L^{\frac{2n}{n-5}}_x} \|w\|_{L^2_t L^{2n}_x}\\
& \lesssim \|D_+ u\|_{L^{\infty}_t H^{\frac{n}{2}-2}_x} \|\Lambda^{\frac{n}{2}-3+8\epsilon} v\|_{H^{2,\half+\epsilon}} \|w\|_{H^{\frac{n}{2}-1,\half+\epsilon}} \\
& \lesssim \|D_+ u\|_{H^{r-1,\half+\epsilon}} \|v\|_{H^{s,\half+\epsilon}}  \|w\|_{H^{s,\half+\epsilon}} \, .
\end{align*}
c. Moreover
\begin{align*}
 \|(\Lambda^{\frac{n}{2}-3+8\epsilon}u)(\partial_{\nu}v) w\|_{L^1_t L^2_x} & \lesssim | \|\Lambda^{\frac{n}{2}-3+8\epsilon} u\|_{L^{\infty}_t L^{\frac{2n}{n-4}}_x}  \| \partial_{\nu} v\|_{L^2_t L^{\frac{2n}{3}}_x}\|w\|_{L^2_t L^{2n}_x}\\
& \lesssim \|D_+ u\|_{L^{\infty}_t H^{\frac{n}{2}-2+8\epsilon}_x} \|\partial_{\nu} v\|_{H^{\frac{n}{2}-2,\half+\epsilon}} \|w\|_{H^{\frac{n}{2}-1,\half+\epsilon}} \\
& \lesssim \|D_+ u\|_{H^{r-1,\half+\epsilon}} \|\Lambda_+v\|_{H^{s-1,\half+\epsilon}} \|w\|_{H^{\frac{n}{2}-1,\half+\epsilon}} \\
& \lesssim \|D_+ u\|_{H^{r-1,\half+\epsilon}} \|\Lambda_+v\|_{H^{s-1,\half+\epsilon}}  \|w\|_{H^{s,\half+\epsilon}} \, .
\end{align*}
d. Next we obtain
\begin{align*}
 \|u(\Lambda^{\frac{n}{2}-3+8\epsilon} \partial_{\nu}v) w\|_{L^1_t L^2_x} & \lesssim \|u\|_{L^2_t L^{2n}_x} \|\Lambda^{\frac{n}{2}-3+8\epsilon}\partial_{\nu} v\|_{L^{\infty}_t L^{\frac{2n}{n-2}}_x} \|w\|_{L^2_t L^{2n}_x}\\
& \lesssim \|D_+ u\|_{H^{\frac{n}{2}-2,\half+\epsilon}} \| \Lambda_+ v\|_{H^{\frac{n}{2}-2+6\epsilon,\half+\epsilon}} \|w\|_{H^{\frac{n}{2}-1,\half+\epsilon}} \\
& \lesssim \|D_+ u\|_{H^{r-1,\half+\epsilon}} \|v\|_{H^{s,\half+\epsilon}}  \|w\|_{H^{s,\half+\epsilon}} \, .
\end{align*}
e. Finally
\begin{align*}
 \|u( \partial_{\nu}v) (\Lambda^{\frac{n}{2}-3+8\epsilon}w)\|_{L^1_t L^2_x} & \lesssim \|u\|_{L^2_t L^{2n}_x} \|\partial_{\nu} v\|_{L^2_t L^{\frac{2n}{3}}_x} \|\Lambda^{\frac{n}{2}-3+8\epsilon} w\|_{L^{\infty}_t L^{\frac{2n}{n-4}}_x}\\
& \lesssim \| u\|_{H^{\frac{n}{2}-1,\half+\epsilon}} \|D_+v\|_{H^{\frac{n}{2}-2,\half+\epsilon}} \|\Lambda^{\frac{n}{2}-1+8\epsilon} w\|_{L^{\infty}_t L^2_x} \\
& \lesssim \|u\|_{H^{r,\half+\epsilon}} \|\Lambda_+ v\|_{H^{s-1,\half+\epsilon}}  \|w\|_{H^{s,\half+\epsilon}} \, .
\end{align*}
2. Now we consider the case $4 \le n \le 5$ . We have to estimate
\begin{align*}
 &\| \Lambda^{\frac{n}{2}-3+8\epsilon} \partial_{\nu}(uvw)\|_{L^1_t L^2_x} \lesssim \|\partial_{\nu}(uvw)\|_{L^1_t L^{\frac{n}{3-8\epsilon}}} \, ,
\end{align*}
where we used Sobolev. By symmetry in $v$ and $w$ we only have to estimate the following two terms. By Sobolev and Strichartz we obtain
\begin{align*}
\|(\partial_{\nu}u)vw)\|_{L^1_t L^{\frac{n}{3-8\epsilon}}}&  \lesssim \|\partial_{\nu} u\|_{L^{\infty}_t L^{\frac{n}{2-8\epsilon}}} \|v\|_{L^2_t L^{2n}_x} \|w\|_{L^2_t L^{2n}_x} \\
& \lesssim \|D_+ u\|_{L^{\infty}_t H^{r-1}} \|v\|_{H^{{\frac{n}{2}-1,\half+\epsilon}}} \|w\|_{H^{\frac{n}{2}-1,\half+\epsilon}} \\
& \lesssim \|D_+ u\|_{H^{r-1,\half+\epsilon}} \|v\|_{H^{s\half+\epsilon}} \|w\|_{H^{s\half+\epsilon}} 
\end{align*}
and
\begin{align*}
\|u(\partial_{\mu}v)w)\|_{L^1_t L^{\frac{n}{3-8\epsilon}}}&  \lesssim  \|u\|_{L^2_t L^{2n}_x} \|\partial_{\mu} v\|_{L^{\infty}_t L^{\frac{n}{2-8\epsilon}}} \|w\|_{L^2_t L^{2n}_x} \\
& \lesssim \|D_+ u\|_{H^{r-1,\half+\epsilon}} \|\Lambda_+v\|_{H^{s-1,\half+\epsilon}} \|w\|_{H^{s\half+\epsilon}} \, ,
\end{align*}
where we remark that  Prop. \ref{Str} remains applicable for homogeneous spaces $\dot{H}^{s,b}$ instead of $H^{s,b}$ .

The proof of Theorem \ref{Theorem2} is now complete.

\section{Proof of Theorem \ref{Theorem1}}
Let $(A,\phi)$ be the solution  of (\ref{16'}),(\ref{17'}) given by Theorem \ref{Theorem2}. Then it is also a solution of (\ref{16}),({\ref{17}) provided the Lorenz condition $\partial^{\mu} A_{\mu} =0$ is satisfied (cf. the remark after Theorem \ref{Theorem2}), what we now prove. We define
$$
u  := \partial^{\mu} A_{\mu} = -\partial_t A_0 + \partial^j A_j \, . $$
By (\ref{16'}) and (\ref{17'}) we obtain
\begin{align*}
\square \, u &= -\partial_t \square \, A_0 + \partial^j \square \, A_{j} \\
&= \partial_t(Im(\phi \overline{\partial_t\phi}) +A_0|\phi|^2)-\partial^j(Im(\phi\overline{\partial_j \phi}) - A_j |\phi|^2) \\
&= Im(\phi \overline{\partial_t^2 \phi}) - Im(\phi \overline{\partial^j \partial_j \phi}) - \partial^{\mu}(A_{\mu} |\phi|^2) \\
&= Im(\phi \overline{\square \, \phi}) -   \partial^{\mu}(A_{\mu} |\phi|^2)) \\
&= Im(\phi \overline{{\tilde M}(A,\phi)}) -\partial^{\mu}(A_{\mu} |\phi|^2)) \\
&=: R(A,\phi)
\end{align*}
By (\ref{17'}) we have
$$
\tilde{M}(A,\phi)  = -2i (A_0 \partial_t + D^{-2}\nabla \partial_t A_0 \cdot \nabla \phi - A^{df} \cdot \nabla \phi )+ A_{\mu} A^{\mu} \phi + m^2 \phi \, .$$
Now by the definition of $A^{df}$ :
\begin{align*}
 A^{df} \cdot \nabla \phi =R^k(R_j A_k - R_k A_j) \partial^j \phi &= D^{-2} \partial^k(\partial_j A_k - \partial_k A_j) \partial^j \phi \\
 &= D^{-2} \nabla(\partial^k A_k)\cdot \nabla \phi + A_j \partial^j \phi 
 \end{align*}
 and by the definition of $u$
$$ D^{-2} \nabla \partial_t A_0 \cdot \nabla \phi = D^{-2} \nabla(\partial^j A_j -u)\cdot \nabla \phi \, , $$
so that we obtain
\begin{align*}
\tilde{M}(A,\phi) & = -2i(A_0 \partial_t \phi - D^{-2} \nabla u \cdot \nabla \phi - A_j \partial^j \phi) +  A_{\mu} A^{\mu} \phi  + m^2 \phi\\
& = 2i(A_{\mu} \partial^{\mu} \phi + D^{-2} \nabla u \cdot \nabla \phi) + A_{\mu} A^{\mu} \phi + m^2 \phi\, ,
\end{align*}
which implies
\begin{align*}
&R(A,\phi) \\
& = Im(-\phi 2i(A_{\mu} \partial^{\mu} \overline{\phi} + D^{-2} \nabla u \cdot \nabla \overline{\phi}) + A_{\mu} \partial^{\mu}(|\phi|^2) +|\phi|^2 u + m^2 \phi\\
& = 2 Re(-\phi A_{\mu} \partial^{\mu} \overline{\phi}) -2 Re(\phi D^{-2} \nabla u \cdot \nabla \overline{\phi}) + 2 Re (A_{\mu} \phi \partial^{\mu} \overline{\phi}) + |\phi|^2 u + m^2 \phi\\
& = -2 Re(\phi \nabla \overline{\phi}) \cdot D^{-2} \nabla u + |\phi|^2 u + m^2 \phi\, ,
\end{align*}
so that $u$ fulfills the linear equation
$$ \square u = -2 Re(\phi \nabla \overline{\phi}) \cdot D^{-2} \nabla u + |\phi|^2 u +m^2 \phi\, . $$
The data of $u$ fulfill by (\ref{10}) and (\ref{12}):
$$ u(0) = -\partial_t A_0(0) + \partial^j A_j(0) = - \dot{a}_{00} + \partial^j a_{0j} = 0 $$
and, using that $A$ is a solution of (\ref{16}) and also (\ref{15}):
\begin{align*}
\partial_t u(0) & = -\partial_t^2 A_0(0) + \partial_t \partial^j A_j(0) = -\partial^j \partial_j A_0(0) - j_{0_{|t=0}} + \partial_t \partial^j A_j(0) \\
& = -\partial^j (\partial_j A_0(0) - \partial_t A_j(0)) - j_{0_{|t=0}} = \partial^j F_{0j}(0) -  j_{0_{|t=0}} \\
& = Im(\phi_0 \overline{\phi}_1) - j_{0_{|t=0}} =  j_{0_{|t=0}}- j_{0_{|t=0}} = 0 \, .
\end{align*}
By uniqueness this implies $u = 0$. Thus the Lorenz condition $\partial^{\mu} A_{\mu} =0$ is satisfied. Under the Lorenz condition however we know that $(A,\phi)$ also fulfills (\ref{16}),(\ref{17}). This system is equivalent to (\ref{1}),(\ref{2}), where $F_{\mu \nu} := \partial_{\mu} A_{\nu} - \partial_{\nu} A_{\mu}$, so that we also have that $F_{k0}$  and $F_{kl}$ satisfy (\ref{2.10a}) and (\ref{2.11a}), respectively.\\[0.5em]

What remains to be shown is the following property of the electromagnetic field $F_{\mu \nu}$:
$$
F_{\mu \nu} \in G^{s-1}_T \, . $$
Using well-known results for Bourgain type spaces (\cite{KS}, Thm. 5.4 and Prop. 5.6)  this is reduced to:
\begin{align}
\label{4.10}
&F_{\mu \nu} (0) \in H^{s-1} \\
\label{4.11}
&\partial_t  F_{\mu \nu}(0) \in H^{s-2} \quad \mbox{and} \\
\label{4.12}
&\Lambda_+^{-1} \Lambda_-^{\epsilon-1} \square \,F_{\mu \nu} \in G^{s-1}_T\, ,
\end{align}
where $\square \,F_{\mu \nu} $ is given by (\ref{2.10a}) and (\ref{2.11a}). \\
[0.5em]
Thus we have to prove the following estimates.
\begin{align}
\label{*}
\|\Lambda_+^{-1} \Lambda_-^{\epsilon-1} Q(u,v)\|_{G^{s-1}} & \lesssim \|u\|_{G^s} \|v\|_{G^s} \\
\label{**}
\|\Lambda_+^{-1} \Lambda_-^{\epsilon-1} D_+ (Auv)\|_{G^{s-1}} & \lesssim \|D_+A\|_{F^{r-1}} \|u\|_{G^s} \|v\|_{G^s} \, ,
\end{align}
where $Q=Q_{0i}$ or $Q=Q_{ij}$ .\\[0.4em]
{\bf Proof of (\ref{*}):} We use (\ref{Q2}), where by symmetry we only have to consider the first three terms.
\begin{align}
\label{*1}
\|  \Lambda_+^{-1} \Lambda_-^{-\half-\epsilon} \Lambda_+(uv)\|_{H^{s-2,\half+\epsilon}} & \lesssim \|\Lambda_+^{\half-2\epsilon} u\|_{H^{s-1,\half+\epsilon}} \|v\|_{H^{s-1,\half+\epsilon}} \\
\label{*2}
\|  \Lambda_+^{-1} \Lambda_-^{-1+\epsilon} \Lambda_+(uv)\|_{H^{s-2,\half+\epsilon}} & \lesssim \|\Lambda_+^{\half-2\epsilon} u\|_{H^{s-1,0}} \|v\|_{H^{s-1,\half+\epsilon}} \\
\label{*3}
\|  \Lambda_+^{-1} \Lambda_-^{-1+\epsilon} \Lambda_+(uv)\|_{H^{s-2,\half+\epsilon}} & \lesssim \|\Lambda_+^{\half-2\epsilon} u\|_{H^{s-1,\half+\epsilon}} \|v\|_{H^{s-1,0}} \, .
\end{align}
(\ref{*1}) reduces to
$$\|uv\|_{H^{s-2,0}} \lesssim \|u\|_{H^{s-\half-2\epsilon,\half+\epsilon}} \|v\|_{H^{s-1,\half+\epsilon}} \, .$$
We apply Prop. \ref{Prop.3.6} with parameters $s_0=2-s$ , $s_1=s-\half-2\epsilon$ , $s_2=s-1$ , so that $s_0+s_1+s_2 = s+\half-2\epsilon > \frac{n}{2}-\half$ and $s_1+s_2=2s-\frac{3}{2}-2\epsilon > n-\frac{7}{2}\ge \half$ for $n \ge 4$ under our assumption $s > \frac{n}{2} -1$ . \\
(\ref{*2}) reduces to
$$\|uv\|_{H^{s-2,-\half+2\epsilon}} \lesssim \|u\|_{H^{s-\half-2\epsilon,0}} \|v\|_{H^{s-1,\half+\epsilon}} \, .$$
We apply Cor. \ref{Cor.3.1} with parameters $s_1=2-s$ , $s_0=s-\half-2\epsilon$ , $s_2=s-1$ , so that $s_1+s_2=1$ . \\
(\ref{*3}) reduces to
$$\|uv\|_{H^{s-2,-\half+2\epsilon}} \lesssim \|u\|_{H^{s-\half-2\epsilon,\half+\epsilon}} \|v\|_{H^{s-1,0}} \, .$$
We again apply Cor. \ref{Cor.3.1} with parameters $s_1=2-s$ , $s_2=s-\half-2\epsilon$ , $s_0=s-1$, so that $s_1+s_2=\frac{3}{2}-2\epsilon$ . \\
{\bf Proof of (\ref{**}):} It suffices to prove
$$ \| Auv\|_{H^{s-1,-\half+2\epsilon}} \lesssim \|D A\|_{H^{r-1,\half+\epsilon}}  \| u\|_{H^{s,\half+\epsilon}} \| v\|_{H^{s,\half+\epsilon}} \, .$$
We may replace $D$ by $\Lambda$ and use Cor. \ref{Cor.3.1} , which gives
\begin{align*}
\| A \phi \psi \|_{H^{s-1,-\half+2\epsilon}} & \lesssim \|A\|_{H^{r,\half+\epsilon}} \|\phi \psi\|_{H^{s-\half,0}} \\
& \lesssim \|A\|_{H^{r,\half+\epsilon}} \|\phi\|_{H^{s,\half+\epsilon}}  \psi\|_{H^{s,\half+\epsilon}} 
\end{align*}
with parameters $s_0=s-\half$ , $s_1=r$ , $s_2=1-s$ for the first estimate, so that $s_0+s_1+s_2 = r+\half > \frac{n-1}{2}$ and $s_0+s_1+s_2+s_1+s_2= 2r-s+\frac{3}{2} >\frac{n}{2}$ , because we assume $r > \frac{n}{2}-1$ and $2r-s > \frac{n-3}{2}$ . For the second estimate choose $s_0=\half-s$ , $s_1=s_2=s$ , so that $s_0+s_1+s_2 = s+\half > \frac{n-1}{2}$ and $s_1+s_2 =2s >\half$ . \\[0.5em]

It remains to prove (\ref{4.10}) and (\ref{4.11}). The property (\ref{4.10}) is given by (\ref{8}). Next we prove (\ref{4.11}). By (\ref{1}) we have
$$\partial_t F_{{0k}_{| t=0}} = - \partial_t F_{{k0}_{|t=0}} = - \partial^l F_{{kl}_{| t=0}} + j_{k_{| t=0}} \, . $$ 
By (\ref{8}) we have $\partial^l F_{{kl}_{| t=0}} \in H^{s-2}$. It remains to prove 
$$ j_{k_{| t=0}} = Im(\phi_0 \overline{\partial_k \phi_0}) + |\phi_0|^2 a_{0k} \in H^{s-2} \, . $$
First we obtain
$$ \|\phi_0 \overline{\partial_k \phi_0} \|_{H^{s-2}} \lesssim \|\phi_0\|_{H^s} \| \partial_k \phi_0\|_{H^{s-1}} < \infty $$
by Prop. \ref{SML}, because $s > \frac{n}{2} -1$ .

Concerning the term $|\phi_0|^2 a_{0k}$ we prove \\
{\bf Claim:}
 \begin{align*} \| |\phi_0|^2 a_{0k}\|_{H^{s-2}} & \lesssim \|\phi_0\|_{H^s}^2 \|Da_{0k}\|_{H^{r-1}} 
\end{align*}
\begin{proof} If $r \le \frac{n}{2}$ we obtain by Prop. \ref{SML} :
$$ \| |\phi_0|^2 a_{0k}\|_{H^{s-2}} \lesssim \| |\phi_0|^2\|_{H^{\frac{n}{2}-2+s-r+}} \|a_{0k}\|_{H^r} \lesssim \| \phi_0\|_{H^{s}}^2 \|a_{0k}\|_{H^r} \, , $$
where the first estimate directly follows from Prop. \ref{SML}, and the last estimate requires
$ s \ge \frac{n}{2}-2+s-r+ $ and $-\frac{n}{2}+2-s+r+2s > \frac{n}{2}$ , which hold by the assumptions $s,r >\frac{n}{2}-1$ . In the case $r > \frac{n}{2}$ we obtain by Prop. \ref{SML}
$$ \| |\phi_0|^2 a_{0k}\|_{H^{s-2}} \lesssim \| |\phi_0|^2\|_{H^{s-2}} \|a_{0k}\|_{H^r} \lesssim \| \phi_0\|_{H^{s}}^2 \|a_{0k}\|_{H^r} \, , $$
assuming $ s > \frac{n}{2}-1$ .
\end{proof}

In the case of high frequencies $\ge 1$ of $a_{0k}$ this is enough for our claim to hold. Otherwise we have equivalent frequencies of $|\phi_0|^2 a_{0k}$ and $|\phi_0|^2$ , so that we obtain:
\begin{align*}
\| |\phi_0|^2 a_{0k}\|_{H^{s-2}} & \lesssim  \|\Lambda^{s-2}(|\phi_0|^2) a_{0k}\|_{L^2} \lesssim \|\Lambda^{s-2}(|\phi_0|^2)\|_{L^2} \|a_{0k}\|_{L^{\infty}} \\
& \lesssim \|\phi_0|^2\|_{H^{s-2}} \|D a_{0k}\|_{H^{\frac{n}{2}-1-}} \lesssim \|\phi_0\|_{H^s}^2 \|Da_{0k}\|_{H^{r-1}}\, ,
\end{align*}
which completes the proof.\\[0.2em]
Moreover
$$\partial_t F_{jk} = \partial_0(\partial_j A_k - \partial_k A_j) = \partial_j(\partial_k A_0 + F_{0k}) - \partial_k(\partial_j A_0 + F_{0j}) = \partial_j F_{0k} - \partial_k F_{0j} \, ,$$
so that by (\ref{8}) we obtain
$$ \partial_t F_{{jk}_{| t=0}} = \partial_j F_{{0k}_{| t=0}} - \partial_k F_{{0j}_{| t=0}} \in H^{s-2} \, . $$
The proof of Theorem \ref{Theorem1} is now complete.

\end{document}